\documentclass{amsart}

\usepackage{fullpage,amsmath,amssymb,mathtools,enumerate,amsthm}
\usepackage{bm}
\usepackage{tikz, graphics}
\usepackage{caption,subcaption,float}
\newtheorem{thm}{Theorem}[section]
\newtheorem{prop}[thm]{Proposition}
\newtheorem{lem}[thm]{Lemma}
\newtheorem{cor}[thm]{Corollary}

\newtheorem{qst}[thm]{Question}

\theoremstyle{definition}

\theoremstyle{remark}
\newtheorem{rmk}[thm]{Remark}

\theoremstyle{definition}

\newcommand{\N}{\mathbb N}
\newcommand{\Z}{\mathbb Z}

\renewcommand{\P}{\mathbb P}

\renewcommand{\phi}{\varphi}

\renewcommand{\ker}{\mathop{\mathrm{Ker}}\nolimits}

\newcommand{\into}{\hookrightarrow}

\renewcommand{\bar}{\protect\overline}

\renewcommand{\v}{\vee}

\newcommand{\n}{\i{n}}
\renewcommand{\a}{\i{a}}

\newcommand{\depth}{\mathop{\mathrm{depth}}\nolimits}

\newcommand{\reg}{\mathop{\mathrm{reg}}\nolimits}

\newcommand{\pd}{\mathop{\mathrm{pd}}\nolimits}

\newcommand{\abs}[1]{\left| #1 \right|}

\newcommand{\mat}[1]{\begin{pmatrix} #1 \end{pmatrix}}

\newcounter{CaseCount}
\newcommand{\resetcase}{\setcounter{CaseCount}{0}}
\newcommand{\case}[1][\theCaseCount]{\stepcounter{CaseCount}\noindent\textbf{Case #1: }}

\title[$N_p$ for Bipartite Toric Edge Ideals]{Green-Lazarsfeld Condition for Toric Edge Ideals of Bipartite Graphs}

\author{Zachary Greif}
\email{zgreif@iastate.edu}
\author{Jason McCullough}
\email{jmccullo@iastate.edu}
\address{Iowa State University, Department of Mathematics, Ames, IA 50011}


\begin{document}

\begin{abstract}  Previously, Ohsugi and Hibi gave a combinatorial description of bipartite graphs $G$ whose toric edge ideal $I_G$ is generated by quadrics, showing that every cycle of $G$ of length at least $6$ must have a chord.  This corresponds to the Green-Lazarsfeld condition $\mathbf{N}_1$.  In this paper, we investigate the higher syzygies of $I_G$ and give combinatorial descriptions of the Green-Lazarsfeld conditions $\mathbf{N}_p$ of toric edge ideals of bipartite graphs for all $p \ge 1$.  In particular, we show that $I_G$ is linearly presented (i.e. satisfies condition $\mathbf{N}_2$) if and only if the bipartite complement of $G$ is a tree of diameter at most $3$.  We also investigate the regularity of linearly presented toric edge ideals and give criteria for polyomino ideals to satisfy the Green-Lazarsfeld conditions.
\end{abstract}

\maketitle

\section{Introduction}

Let $k$ be a field and let $G = (V,E)$ be a finite, simple graph.  Let $k[V]$ denote the polynomial ring with variables corresponding to the vertices of $G$.  The edge ring $k[G]$ of $G$ is the $k$-subalgebra of $k[G]$ generated by the quadratic monomials corresponding to the edges of $G$.  The toric edge ideal $I_G$ is the presenting ideal of $k[G]$ in the polynomial ring $k[E]$ whose variables correspond to the edges of $G$.  In particular, $I_G$ is a homogeneous prime ideal generated by binomials.    When $G$ is a complete bipartite graph, $I_G$ defines a Segre embedding.  Such ideals are special cases of toric ideals where one finds defining equations of ideals generated by any set of monomials; the restriction to toric edge ideals corresponds to considering only subrings generated by squarefree monomials of degree two.  There has been significant interest in understanding the minimal free resolutions of $I_G$ for different classes of graphs; see e.g. \cite{OH1, OH2, BOV, HKO}

%a polynomial ring The subring $k[G] = k[x_{i_1}x_{i_2}\,|\,e_i = \{x_1,x_2\} \in E] \subset k[V]$ is called the edge subring of $G$.  We define the projection $\pi:k[E] \to k[G]$ by mapping $e_i \mapsto x_{i_1}x_{i_2}$.  The kernel of $\pi$ is denoted $I_G$ and is called the toric edge ideal of $G$.  It is well-known that the generators $I_G$ correspond to even, closed walks in $G$.  (See \cite[Proposition 10.1.5]{V}.). 

If $G$ is a bipartite graph, then more is known about $I_G$.  The following result is due to Ohsugi and Hibi:

\begin{thm}[{\cite[Theorem 1]{OH2}}]\label{OHthm}
Let $G$ be a bipartite graph.  The following are equivalent:
\begin{enumerate}
\item Every cycle in $G$ of length $\ge 6$ has a chord.
\item $I_G$ has a Gr\"obner basis consisting of quadratic binomials.
\item $k[G]$ is Koszul.
\item $I_G$ is generated by quadratic binomials, corresponding to the $4$-cycles of $G$.
\end{enumerate}
\end{thm}

%In light of this theorem, we say that a bipartite graph is chordal if every cycle in $G$ of length $\ge 6$ has a chord.  
One can generalize the property of having a quadratic generating set by considering the degrees of syzygies of $I_G$ over $S$.  The Green-Lazarsfeld condition $\mathbf{N_p}$ describes ideals (defining normal quotient rings) generated by quadrics with linear syzygies for the first $p-1$ steps of the resolution.   If $G$ is a bipartite graph, then Theorem~\ref{OHthm} says that $I_G$ satisfies property $\mathbf{N}_{1}$ if and only if every cycle in $G$ of length $\geq 6$ has a chord. The main goal of this paper is to give a combinatorial description of when $I_G$ satisfies property $\mathbf{N}_{p}$ for all $p \ge 0$.   We first need a couple definitions to state our main result.

%View $S = \bigoplus\limits_{i \ge 0} S_i$ as a standard graded ring in which all the variables have degree one and $S_i$ is the $k$-vector space of homogeneous degree $i$ polynomials.  Let $\mathbf{F}_\bullet$ denote the minimal graded free resolution of $I_G$ over $S$ so that $F_i = \bigoplus\limits_{j} S(-j)^{\beta_{ij}}$, where $S(-j)_i = S_{i-j}$.  Following \cite{EGHP}, we say that $I_G$ satisfies condition $\mathbf{N}_{p}$  if $S/I_G$ is normal and $\beta_{i,j} = 0$ for all $i,j$ such that $i < p$ and $j > d+j$.  Thus property $\mathbf{N}_0$ corresponds to projective normality; property $\mathbf{N}_{1}$ corresponds to ideals that in addition to satisfying $\mathbf{N}_0$ are generated by quadrics; property $\mathbf{N}_2$ corresponds to the additional requirement that all first syzygies are linear.  There has been a lot of interest in calculating the Green-Lazarsfeld index for classes of varieties.  CITE STUFF HERE  One reason for the attention is that  the ideal $I_G$ has a linear free resolution exactly when it satisfies $\mathbf{N}_{p}$ for all $p \ge 0$.  Thus theorem~\ref{OHthm} can be restated as follows:

Let $G = (V,E)$ be a bipartite graph and let $V = X \sqcup Y$ be a partition of $V$ so that $X = \{x_1,\ldots,x_m\}$, $Y = \{y_1,\ldots,y_n\}$ and all edges $e \in E$ are of the form $e = \{x_i, y_j\}$ for some $1 \le i \le m$ and $1 \le j \le n$.    
%The degree of a vertex $v \in V$ is the number of edges incident to $v$.  If $G$ is disconnected, say $G = G_1 \sqcup G_2$, then $k[G] = k[G_1] \otimes_k k[G_2]$, so we assume that $G$ is connected throughout.  Since generators of $I_G$ correspond to even cycles, if $\deg(v) = 0$, we may remove $v$ without changing $I_G$.  If $\deg(v) = 1$, then we may also remove $v$ and the incident edge without changing $I_G$ in the sense that $I_{G'}k[G] = I_G$, where $G'$ is the induced subgraph of $G$ with vertex $v$ removed.  So we also can assume $\deg(v) \ge 2$ for all $v \in V$.    
We define the \emph{bipartite complement} of $G$ as the bipartite graph $\bar{G} = (X \sqcup Y, E')$, where $E' = (X \times Y) \setminus E$, viewed as sets.  A graph is \emph{essentially a tree} if it is a tree after perhaps removing some isolated vertices; for a formal definition, see the following section. 

Our main theorem is a  combinatorial characterization of toric ideals of bipartite graphs which satisfy property $\mathbf{N}_p$ for arbitrary $p \ge 1$.

%\begin{thm}\label{main}
%Let $G$ be a connected, bipartite graph with $\deg(v) \ge 2$ for all $v \in V$ and let $k$ be a field.  Then $I_G$ satisfies property $\mathbf{N}_{2}$ if and only if $\bar{G}$ is a tree of diameter at most $3$.
%\end{thm}
%The proof is combinatorial and relies on the simplicial homology formula for multigraded Betti numbers of toric ideals of Aramova and Herzog as well as some Gr\"obner bases techniques.  The intermediate step is a purely graph-theoretic result characterizing certain chordal bipartite graphs whose bipartite complement is a tree of diameter at most 3.

%As a result, we also get characterizations of bipartite graphs satisfying properties $\mathbf{N}_{p}$ for all other $p$. 

\begin{thm}\label{main2}
Let $G$ be a  bipartite graph with minimum vertex degree at least $2$ and let $k$ be a field. %CHANGE TO \delta(G)?
\begin{enumerate}
\item $I_G$ satisfies property $\mathbf{N}_1$ if and only if every cycle of length $\ge 6$ has a chord.
\item $I_G$ satisfies property $\mathbf{N}_2$ if and only if $\bar{G}$ is essentially a tree of diameter at most $3$.
\item $I_G$ satisfies property $\mathbf{N}_{3}$ if and only if $G$ is a complete bipartite graph unless the characteristic of $k$ is 3 and $G = K_{m,n}$  with $\min\{m, n\} \ge 5$.
\item $I_G$ satisfies property $\mathbf{N}_{p}$ for some/any $p \ge 4$ if and only if $G = K_{2,n}$ for some $n$.
\end{enumerate}
\end{thm}

\noindent The proof of Theorem~\ref{main2} is given in the proofs of Theorems~\ref{N2entire}, \ref{N3}, and \ref{N4}.

The rest of this paper is organized as follows: Section~\ref{pre} sets notation and basic definitions.  Section~\ref{obstruct} contains our main tools for finding obstructions to vanishing of graded Betti numbers.    Section~\ref{graph} gives a purely graph-theoretic result we need to connect local and global graph structure. Our main results appear in Section~\ref{main}.  In Section~\ref{poly}, we also obtain a characterization of the Green-Lazarsfeld conditions for ideals associated to convex polyominoes.  This seems to correct an omission in the characterization of linearly presented polyomino ideals in \cite{EHH}.  Finally, in Section~\ref{reg} we apply our result to a special case of a recent question of Constantinescu, Kahle, and Varbaro \cite{CKV}.

%%%%%%%%%%%%%%%%%%%%%%%%%%%%%%%%%%%%%%%%%%%%%%%%%%%%%%%%%%%%%%%%%%%%%%%%%%%%%%%%%%%%%%%%%%%%%%
\section{Preliminaries}\label{pre}

Here we fix notation for the remainder of the paper.  We first record the standard graph-theoretic definitions we require.

\subsection{Graph Theory}

All graphs considered in this paper are finite and simple.  Let $G = (V,E)$ be a finite simple graph with vertex set $V$ and edge set $E$.  The graph $G$ is bipartite if there is a partition $V = X \sqcup Y$ such that all edges in $E$ lie in $X \times Y$; that is, all edges contain one vertex  in $X$ and one vertex  in $Y$.  For positive integers $m,n$, the complete bipartite graph $K_{m,n}$ has vertex set $V = X \sqcup Y$ with $|X| = m$, $|Y| = n$ and edge set $E = X \times Y$.  The degree of a vertex is the number of edges incident to it.  The minimum degree of a vertex in a graph $G$ is denoted $\delta(G)$.  An isolated vertex is a vertex of degree $0$.  A path of length $t$ from vertex $v$ to vertex $w$ is a sequence of vertices $v = v_0, v_1, \ldots, v_t = w$ such that $\{v_{i-1},v_i\} \in E$ for all $1 \le i \le t$.  A graph is connected if for any two vertices $v,w \in V$, there is a path from $v$ to $w$.    A cycle of length $t$ in $G$ is a path of length $t$ from $v$ to itself.    Such a cycle has a chord if $\{v_i, v_j\} \in E$ for distinct $1 \le i, j \le t$.  A graph $G$ is a tree if there is a unique path between any two distinct vertices of $G$, or equivalently, $G$ is a tree if it is connected and has no cycles.  %A bipartite graph is called chordal bipartite if every cycle of length at least $6$ has a chord.  (Note that chordal bipartite graphs are not themselves chordal.)
Given a subset $W \subseteq V$, the induced graph $G_W$ is the graph with vertex set $W$ and edge set given by all edges of $G$ both of whose vertices lie in $W$.   The diameter of a graph is the minimum integer $n$ such that for any pairs of vertices $v, w \in V$, there is a path of length at most $n$ starting at $v$ and ending at $w$. A perfect matching in a graph is a collection $M \subseteq E$ such that every vertex is incident to exactly one edge in $M$.  A bridge is an edge whose deletion increases the number of connected components.

We add some new graph-theoretic definitions to the standard definitions above.   For a nonnegative integer $k$ and graph $G$, we define the \textbf{degree $k$ subgraph} of $G$ to be the largest induced subgraph $G_k$ such that all vertices have degree at least $k$.  Thus $G_0 = G$; $G_1$ is the subgraph of $G$ with all isolated vertices removed.  
 For a graph property $P$, we say that a graph is \textbf{essentially P} if $G_1$ satisfies property $P$.  Thus a graph $G$ is called \textbf{essentially a tree} if $G_1$ is a tree.   If $G = (X \sqcup Y, E)$ is a bipartite graph, then the \textbf{bipartite complement} of $G$, denoted $\bar{G}$, is the graph with same vertex set $X \sqcup Y$ and with edge set $(X \times Y) \setminus E$; that is, an edge $\{x,y\}$ with $x \in X$ and $y \in Y$ is in $\bar{G}$ if and only if it is not in $G$.  %Given a bipartite graph $G = (X \sqcup Y, E)$, an \textbf{$\boldsymbol{(n,m)}$-induced subgraph} of $G$ is an induced subgraph on vertex set $X' \sqcup Y'$ with $X' \subset X$, $Y' \subset Y$, $|X'| = n$ and $|Y'| = m$. 

\subsection{Toric Edge Ideals} Let $G = (V,E)$ be a finite simple graph, with $V = \{v_1,\ldots,v_n\}$ and fix a field $k$.  By abuse of notation, we also view the $v_i$ has variables in the polynomial ring $k[V] = k[v_1,\ldots,v_n]$.  The \textbf{edge ring} of $G$, denoted $k[G]$, is the  $k$-subalgebra of $k[V]$ generated by $v_iv_j$, where $\{v_i, v_j\} \in E$.   In the special case we focus on, where $G$ is a bipartite graph, we denote by $V = X \sqcup Y$ the partition of the vertex set, where $X = \{x_1,\ldots,x_m\}$ and $Y = \{y_1,\ldots,y_n\}$ and all edges (viewed as ordered pairs) are contained in $X \times Y$.  Denote by $S = k[e_{i,j}\,|\,\{x_i,y_j\} \in E]$ a polynomial ring with variables $e_{i,j}$ corresponding to the edges in $G$.   The surjective map $\pi: S = k[e_{i,j}\,|\, \{x_i,y_j\} \in E] \to k[G]$ sends $e_{i,j} \mapsto x_iy_j$.    The ideal $I_G = \ker(\pi)$ is called the \textbf{toric edge ideal} of $G$.   For an arbitrary graph $G$,  it is well-known that the generators of $I_G$ are binomials corresponding to even closed walks of $G$ \cite[Lemma 5.9]{HHO}.  When $G$ is bipartite, the ring $S/I_G \cong k[G]$ is Cohen-Macaulay \cite[Corollary  5.26]{HHO}.  Ohsugi and Hibi \cite{OH2} gave the characterization in Theorem~\ref{OHthm} of bipartite graphs for which $I_G$ is generated by quadratic binomials.  It follows that all such rings are normal \cite[Corollary 5.25]{HHO}.  It is then natural to investigate the properties of the syzygies of such ideals.  

\subsection{Green-Lazarsfeld Conditions}

Unless otherwise noted, we regard $S$ as a standard graded ring with $\deg(e_{i,j}) = 1$ for all $i,j$.  Writing $S(-j)$ for the rank-one free $S$-module with $S(-j)_i = S_{i-j}$, we consider the minimal graded free resolution of $S/I_G$:
\[0 \to \bigoplus_{j} S(-j)^{\beta_{h,j}} \to \cdots \to \bigoplus_j S(-j)^{\beta_{1,j}} \to S.\]
Here $\beta_{i,j}$ denotes the minimal, graded Betti numbers of $S/I_G$, which by the uniqueness of minimal, graded free resolutions, are invariants of $S/I_G$.  The projective dimension is $\pd_S(S/I_G) = \max\{i\,|\,\beta_{i,j} \neq 0\} = h$ and the regularity is $\reg(S/I_G) = \max\{j - i\,|\,\beta_{i,j} \neq 0\}$.  We often refer to the graded Betti numbers of $I_G$, noting that $\beta_{i,j}(S/I_G) = \beta_{i-1,j}(I_G)$, and therefore $\pd_S(S/I_G) = \pd_S(I_G) + 1$ and $\reg(S/I_G) = \reg(I_G) - 1$.  

With the notation above, we say that $I_G$ satisfies condition $\mathbf{N}_p$ if $S/I_G$ is (projectively) normal and $\beta_{i,j}(I_G) = 0$ for $i < p$ and $j > i + 2$.  Thus condition $\mathbf{N}_0$ means that $S/I_G$ is normal; condition $\mathbf{N}_1$ means that in addition to $\mathbf{N}_0$, $I_G$ is generated by quadrics; condition $\mathbf{N}_2$ means that in addition to satisfying $\mathbf{N}_1$, $I_G$ is linearly presented; and so on.  This idea was first defined by Green and Lazarsfeld \cite{GL1, GL2}.  The $\mathbf{N}_p$ conditions and their generalizations have been well studied; see for example \cite{EGHP, KP}.

Note that in the specific case that $G = K_{m,n}$, the ideal $I_G$ defines to the image of the Segre embedding of $\P^{m-1}_k \times \P^{n-1}_k \into \P^{mn-1}_k$ whose resolutions in characteristic 0 are known by work of Pragacz-Weyman \cite{PW} and Lascoux \cite{L}; see also Roberts \cite{R}.  If $\min\{m,n\} \le 4$, Hashimoto and Kurano showed that the Betti numbers of $I_G$ do not depend on the characteristic \cite{HK}.  In particular, this includes $K_{2,n}$ whose toric edge ideal $I_{K_{2,n}}$ is resolved by the linear Eagon-Northcott resolution in all characteristics.  For all $m,n$, the second Betti numbers $\beta_{2,i}(S/I_{K_{m,n}})$ are also independent of the characteristic \cite{HK}.  However, in characteristic $3$, Hashimoto \cite{H} showed that $\beta_{3,i}(S/I_{K_{m,n}})$ does depend on the characteristic of the base field when $m,n \ge 5$.  In this paper, we give a complete description of the Green-Lazarsfeld conditions for bipartite toric edge ideals.  It follows from \cite{L, PW, R} that the precise $\mathbf{N}_p$ conditions for complete bipartite graphs in characteristic 0 are known; see \cite{R} for a summary.

When $I_G$ is the toric edge ideal of a bipartite graph, Ohsugi and Hibi \cite[Theorem 1.1]{OH2} proved that $I_G$ is generated by quadratic binomials (i.e. satisfies condition $\mathbf{N}_1$) if and only if every cycle in $G$ of length at least $6$ has a chord.  Ohsugi and Hibi \cite[Theorem 4.6]{OH1} also showed that $I_G$ has a linear free resolution (i.e. satisfies condition $\mathbf{N}_p$ for all $p$) if and only if $G = K_{2,n}$ for some $n$.   Thus our main theorem interpolates between these two results.  In related work, Hibi, Matsuda, and Tsuchiya \cite{HMT} show that the only toric edge ideals with $3$-linear resolutions are hypersurfaces. 

%%%%%%%%%%%%%%%%%%%%%%%%%%%%%%%%%%%%%%%%%%%%%%%%%%%%%%%%%%%%%%%%%

\section{Obstructions to Vanishing of Graded Betti Numbers}\label{obstruct}

In this section we prove that the nonvanishing of certain graded Betti numbers of the toric edge ideal of a graph $G$ correspond in a precise way to forbidden induced subgraphs of $G$.  A version of this result was proved previously by Ha, Kara, and O'Keefe in \cite[Theorem 3.6]{HKO}.  Our result quantifies how large the forbidden subgraph must be relative to the index of the graded Betti number in question.  Toward this end, we follow the notation in \cite{P}.  Let $G = (V,E)$ be a finite simple graph on vertex set $V = \{v_1,\ldots,v_n\}$.  Given a field $k$, the edge ring $k[G] = k[v_iv_j\,|\,\{v_i,v_j\} \in E]$ is a subring of the polynomial ring $k[V]$, which we view as a multigraded ring by setting ${\rm mdeg}(v_i) = \underline{e}_i$, where $\underline{e}_i$ denotes the $i$th standard basis vector of $\Z^n$.  By setting $k[E] = k[e_{ij}\,|\,\{v_i,v_j\} \in E]$ to be the multigraded ring with ${\rm mdeg}(e_{ij}) = {\rm mdeg}(v_iv_j) = \underline{e}_i + \underline{e}_j$, the toric edge ideal $I_G \subset k[E]$ is also multigraded.  Fix a multidegree $\underline{\alpha}$.  The fiber of $\underline{\alpha}$, denoted $C_{\underline{\alpha}}$, is the set of all monomials of $k[E]$ of multidegree $\underline{\alpha}$.  We let $\Gamma(\underline{\alpha})$ denote the simplicial complex associated to $\underline{\alpha}$ with vertices identified with the variables $e_{ij}$  and whose faces are identified with the radicals of monomials in $\Gamma(\underline{\alpha})$.  With this notation, we have the following result of Aramova and Herzog:

\begin{thm}[cf. {\cite[Theorem 67.5]{P}}]
For $\underline{\alpha} \in \mathbb{N}^n$ and $i \ge 0$ we have
\[\beta_{i,\underline{\alpha}}(I_G) = \dim_k \widetilde{H}_i(\Gamma(\underline{\alpha});k).\]
\end{thm}

\noindent Here $\widetilde{H}_i(\Delta;k)$ denotes the reduced simplicial homology of the simplicial complex $\Delta$ with coefficients in $k$.  Comparing the standard grading on $I_G$ with the multigrading, we see that 
\[\beta_{i,j}(I_G) = \sum_{\substack{\underline{\alpha}\\\sum \underline{\alpha} = 2j}} \beta_{i,\underline{\alpha}}(I_G).\]
This perspective gives us a way of finding local obstructions to the vanishing of certain graded Betti numbers of $I_G$.  

\begin{thm}\label{InducedGraphBetti}
 Let $G$ be a graph with toric edge ideal $I_G$.  Then $\beta_{i,j}(I_G) \neq 0$ if and only if there is an induced subgraph $H$ of $G$ with at most $2j$ vertices such that $\beta_{i,j}(I_H) \neq 0$.
\end{thm}

\begin{proof}  
  Let $G = (V, E)$ with $V = \{v_1, \hdots, v_n\}$ and $E = \{e_1, \hdots, v_r\}$. Suppose $\beta_{i, j}(I_G) \neq 0$. Then $\beta_{i, \underline{\alpha}}(I_G) \neq 0$ for some multidegree $\underline{\alpha}$ such that $\sum_{\ell} \alpha_\ell = 2j$. Let $V' = \{v_\ell \in V : \alpha_\ell \neq 0\}$. Since at most $2j$ of the $\alpha_\ell$ are nonzero, we have $\abs{V'} \leq 2j$. Let $H = (V', E')$ be the induced subgraph of $G$ on $V'$. $k[E']$ is a subring of $S$, so it is $\Z^n$-graded. Let $C_{\underline{\alpha}}^G$ and $C_{\underline{\alpha}}^H$ denote the fibers of $\alpha$ in $k[E]$ and $k[E']$ respectively. Since $k[H]$ is a subring of $k[G]$, we know $C_{\underline{\alpha}}^H \subseteq C_{\underline{\alpha}}^G$.

  Suppose $f = e_1^{m_1}e_2^{m_2}\cdots e_r^{m_r}\in C_{\underline{\alpha}}^G \setminus C_{\underline{\alpha}}^H$. Since $f \notin k[E']$, $m_\ell > 0$ for some $\ell$ such that $e_\ell \notin E'$. Since $H$ is induced, $e_{\ell} = \{v_{\ell_1}, v_{\ell_2}\}$ where at least one of $v_{\ell_1}$ and $v_{\ell_2}$ is not in $V'$. Suppose, w.l.o.g. $v_{\ell_1} \notin V'$. Since $f$ has multidegree $\underline{\alpha}$, we have $\alpha_{\ell_1} \neq 0$, giving that $v_{\ell_1} \in V'$, a contradiction. Thus no such $f$ exists and $C_{\underline{\alpha}}^G = C_{\underline{\alpha}}^H$. Then the associated simplicial complexes are the same and we must have $\beta_{i, \underline{\alpha}}(I_H) \neq 0$, giving $\beta_{i,j}(I_H) \neq 0$.
\end{proof}

It follows that to characterize when a particular $\beta_{i,j}(I_G) = 0$, we could simply enumerate over all graphs $H$ of at most $2j$ vertices for which $\beta_{i,j}(I_H) \neq 0$ and then check to see if any such $H$ is an induced subgraph of $G$.  While this strategy would work, it is not very efficient.  In the next section adopt a more efficient strategy taking advantage of the fact that quadratic toric edge rings of bipartite graphs are Koszul.  However, we do take advantage of the previous theorem by identifying a few key induced subgraphs that act as obstructions to satisfying the $N_p$ property for various $p$.  

There are 8 bipartite graphs whose presence as an induced subgraph characterizes failure of $I_G$ satisfying $\mathbf{N}_2$ for $G$ a bipartite graph and such that every cycle of length at least 6 has a chord. These 8 forbidden graphs are those pictured in Figure~\ref{3cases}.

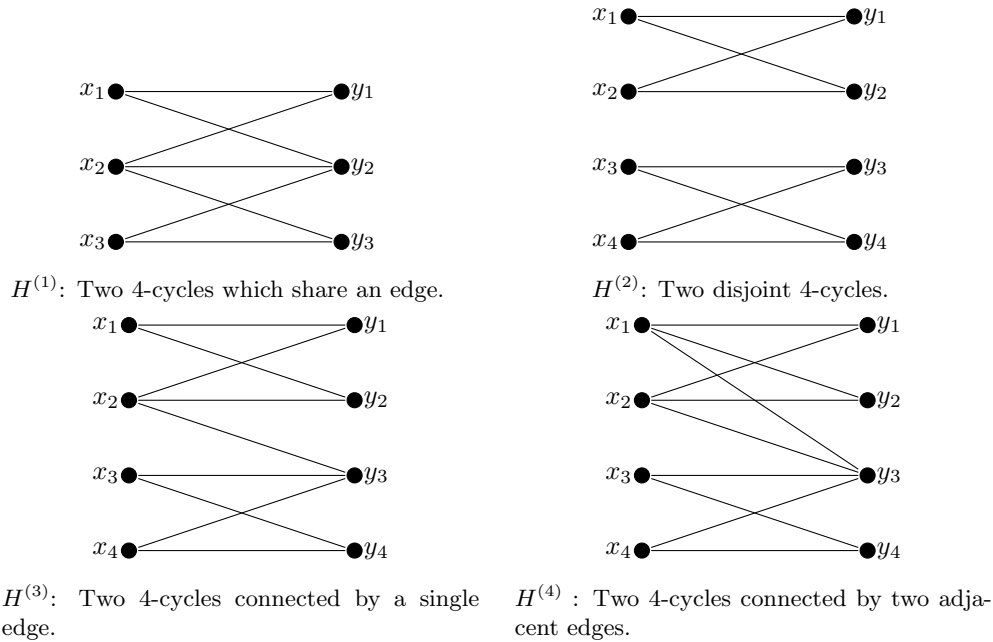
\begin{figure}[H]
    \centering
    \begin{subfigure}[t]{2.5in}
      \centering
      \begin{tikzpicture}
        \def\hor{3cm}
        \def\v{-1cm}

        \draw[fill=black] (0*\hor, 0*\v) circle (.1cm) node (v1) {} node[left] {$x_1$};
        \draw[fill=black] (1*\hor, 0*\v) circle (.1cm) node (u1) {} node[right] {$y_1$};
        \draw[fill=black] (0*\hor, 1*\v) circle (.1cm) node (v2) {} node[left] {$x_2$};
        \draw[fill=black] (1*\hor, 1*\v) circle (.1cm) node (u2) {} node[right] {$y_2$};
        \draw[fill=black] (0*\hor, 2*\v) circle (.1cm) node (v3) {} node[left] {$x_3$};
        \draw[fill=black] (1*\hor, 2*\v) circle (.1cm) node (u3) {} node[right] {$y_3$};

        \draw (v1) -- (u1);
        \draw (v1) -- (u2);
        \draw (v2) -- (u1);
        \draw (v2) -- (u2);
        \draw (v2) -- (u3);
        \draw (v3) -- (u2);
        \draw (v3) -- (u3);
      \end{tikzpicture}
      \caption*{$H^{(1)}$: Two 4-cycles which share an edge.}\label{3case1}
    \end{subfigure}
\quad
    \begin{subfigure}[t]{2.5in}
      \centering
      \begin{tikzpicture}
        \def\hor{3cm}
        \def\v{-1cm}

        \draw[fill=black] (0*\hor, 0*\v) circle (.1cm) node (v1) {} node[left] {$x_1$};
        \draw[fill=black] (1*\hor, 0*\v) circle (.1cm) node (u1) {} node[right] {$y_1$};
        \draw[fill=black] (0*\hor, 1*\v) circle (.1cm) node (v2) {} node[left] {$x_2$};
        \draw[fill=black] (1*\hor, 1*\v) circle (.1cm) node (u2) {} node[right] {$y_2$};
        \draw[fill=black] (0*\hor, 2*\v) circle (.1cm) node (v3) {} node[left] {$x_3$};
        \draw[fill=black] (1*\hor, 2*\v) circle (.1cm) node (u3) {} node[right] {$y_3$};
        \draw[fill=black] (0*\hor, 3*\v) circle (.1cm) node (v4) {} node[left] {$x_4$};
        \draw[fill=black] (1*\hor, 3*\v) circle (.1cm) node (u4) {} node[right] {$y_4$};

        \draw (v1) -- (u1);
        \draw (v1) -- (u2);
        \draw (v2) -- (u1);
        \draw (v2) -- (u2);
        \draw (v3) -- (u3);
        \draw (v3) -- (u4);
        \draw (v4) -- (u3);
        \draw (v4) -- (u4);

        %\draw (v1) -- (u3);
        %\draw (v1) -- (u4);
        %\draw (v2) -- (u3);
        %\draw (v2) -- (u4);
        %\draw (v3) -- (u2);
      \end{tikzpicture} 
      \caption*{$H^{(2)}$: Two disjoint 4-cycles.}\label{3case2}
    \end{subfigure} \\
 \quad
    \begin{subfigure}[t]{2.5in}
      \centering
      \begin{tikzpicture}
        \def\hor{3cm}
        \def\v{-1cm}

        \draw[fill=black] (0*\hor, 0*\v) circle (.1cm) node (v1) {} node[left] {$x_1$};
        \draw[fill=black] (1*\hor, 0*\v) circle (.1cm) node (u1) {} node[right] {$y_1$};
        \draw[fill=black] (0*\hor, 1*\v) circle (.1cm) node (v2) {} node[left] {$x_2$};
        \draw[fill=black] (1*\hor, 1*\v) circle (.1cm) node (u2) {} node[right] {$y_2$};
        \draw[fill=black] (0*\hor, 2*\v) circle (.1cm) node (v3) {} node[left] {$x_3$};
        \draw[fill=black] (1*\hor, 2*\v) circle (.1cm) node (u3) {} node[right] {$y_3$};
        \draw[fill=black] (0*\hor, 3*\v) circle (.1cm) node (v4) {} node[left] {$x_4$};
        \draw[fill=black] (1*\hor, 3*\v) circle (.1cm) node (u4) {} node[right] {$y_4$};

        \draw (v1) -- (u1);
        \draw (v1) -- (u2);
        \draw (v2) -- (u1);
        \draw (v2) -- (u2);
        \draw (v3) -- (u3);
        \draw (v3) -- (u4);
        \draw (v4) -- (u3);
        \draw (v4) -- (u4);

        %\draw (v1) -- (u3);
        %\draw (v1) -- (u4);
        \draw (v2) -- (u3);
        %\draw (v2) -- (u4);
        %\draw (v3) -- (u2);
      \end{tikzpicture}
      \caption*{$H^{(3)}$: Two 4-cycles connected by a single edge.}\label{3case3}
    \end{subfigure}
 \quad
    \begin{subfigure}[t]{2.5in}
      \centering
      \begin{tikzpicture}
        \def\hor{3cm}
        \def\v{-1cm}

        \draw[fill=black] (0*\hor, 0*\v) circle (.1cm) node (v1) {} node[left] {$x_1$};
        \draw[fill=black] (1*\hor, 0*\v) circle (.1cm) node (u1) {} node[right] {$y_1$};
        \draw[fill=black] (0*\hor, 1*\v) circle (.1cm) node (v2) {} node[left] {$x_2$};
        \draw[fill=black] (1*\hor, 1*\v) circle (.1cm) node (u2) {} node[right] {$y_2$};
        \draw[fill=black] (0*\hor, 2*\v) circle (.1cm) node (v3) {} node[left] {$x_3$};
        \draw[fill=black] (1*\hor, 2*\v) circle (.1cm) node (u3) {} node[right] {$y_3$};
        \draw[fill=black] (0*\hor, 3*\v) circle (.1cm) node (v4) {} node[left] {$x_4$};
        \draw[fill=black] (1*\hor, 3*\v) circle (.1cm) node (u4) {} node[right] {$y_4$};

        \draw (v1) -- (u1);
        \draw (v1) -- (u2);
        \draw (v2) -- (u1);
        \draw (v2) -- (u2);
        \draw (v3) -- (u3);
        \draw (v3) -- (u4);
        \draw (v4) -- (u3);
        \draw (v4) -- (u4);

        \draw (v1) -- (u3);
        %\draw (v1) -- (u4);
        \draw (v2) -- (u3);
        %\draw (v2) -- (u4);
        %\draw (v3) -- (u2);
      \end{tikzpicture}
      \caption*{$H^{(4)}$ : Two 4-cycles connected by two adjacent edges.}\label{3case4}
    \end{subfigure} \\
\end{figure}

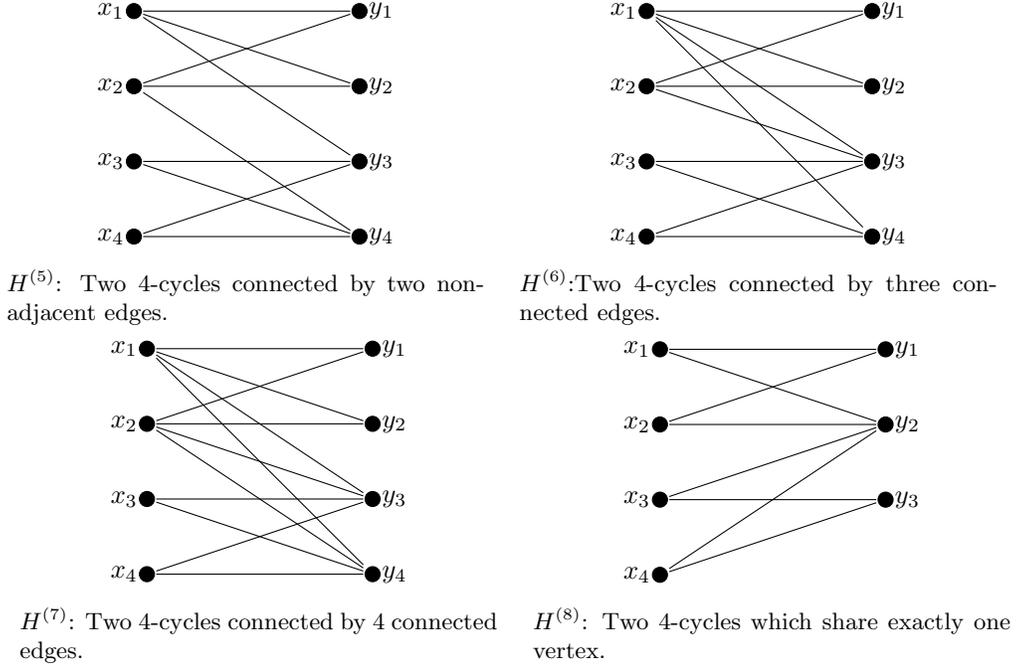
\begin{figure}[H]\ContinuedFloat
    \begin{subfigure}[t]{2.5in}
      \centering
      \begin{tikzpicture}
        \def\hor{3cm}
        \def\v{-1cm}

        \draw[fill=black] (0*\hor, 0*\v) circle (.1cm) node (v1) {} node[left] {$x_1$};
        \draw[fill=black] (1*\hor, 0*\v) circle (.1cm) node (u1) {} node[right] {$y_1$};
        \draw[fill=black] (0*\hor, 1*\v) circle (.1cm) node (v2) {} node[left] {$x_2$};
        \draw[fill=black] (1*\hor, 1*\v) circle (.1cm) node (u2) {} node[right] {$y_2$};
        \draw[fill=black] (0*\hor, 2*\v) circle (.1cm) node (v3) {} node[left] {$x_3$};
        \draw[fill=black] (1*\hor, 2*\v) circle (.1cm) node (u3) {} node[right] {$y_3$};
        \draw[fill=black] (0*\hor, 3*\v) circle (.1cm) node (v4) {} node[left] {$x_4$};
        \draw[fill=black] (1*\hor, 3*\v) circle (.1cm) node (u4) {} node[right] {$y_4$};

        \draw (v1) -- (u1);
        \draw (v1) -- (u2);
        \draw (v2) -- (u1);
        \draw (v2) -- (u2);
        \draw (v3) -- (u3);
        \draw (v3) -- (u4);
        \draw (v4) -- (u3);
        \draw (v4) -- (u4);

        \draw (v1) -- (u3);
        %\draw (v1) -- (u4);
        %\draw (v2) -- (u3);
        \draw (v2) -- (u4);
        %\draw (v3) -- (u2);
      \end{tikzpicture}
      \caption*{$H^{(5)}$: Two 4-cycles connected by two non-adjacent edges.}\label{3case5}
    \end{subfigure}
     \quad
    \begin{subfigure}[t]{2.5in}
      \centering
      \begin{tikzpicture}
        \def\hor{3cm}
        \def\v{-1cm}

        \draw[fill=black] (0*\hor, 0*\v) circle (.1cm) node (v1) {} node[left] {$x_1$};
        \draw[fill=black] (1*\hor, 0*\v) circle (.1cm) node (u1) {} node[right] {$y_1$};
        \draw[fill=black] (0*\hor, 1*\v) circle (.1cm) node (v2) {} node[left] {$x_2$};
        \draw[fill=black] (1*\hor, 1*\v) circle (.1cm) node (u2) {} node[right] {$y_2$};
        \draw[fill=black] (0*\hor, 2*\v) circle (.1cm) node (v3) {} node[left] {$x_3$};
        \draw[fill=black] (1*\hor, 2*\v) circle (.1cm) node (u3) {} node[right] {$y_3$};
        \draw[fill=black] (0*\hor, 3*\v) circle (.1cm) node (v4) {} node[left] {$x_4$};
        \draw[fill=black] (1*\hor, 3*\v) circle (.1cm) node (u4) {} node[right] {$y_4$};

        \draw (v1) -- (u1);
        \draw (v1) -- (u2);
        \draw (v2) -- (u1);
        \draw (v2) -- (u2);
        \draw (v3) -- (u3);
        \draw (v3) -- (u4);
        \draw (v4) -- (u3);
        \draw (v4) -- (u4);

        \draw (v1) -- (u3);
        \draw (v1) -- (u4);
        \draw (v2) -- (u3);
        %\draw (v2) -- (u4);
        %\draw (v3) -- (u2);
      \end{tikzpicture}
      \caption*{$H^{(6)}$:Two 4-cycles connected by three connected edges.}\label{3case6}
    \end{subfigure} \\
     \quad
    \begin{subfigure}[t]{2.5in}
      \centering
      \begin{tikzpicture}
        \def\hor{3cm}
        \def\v{-1cm}

        \draw[fill=black] (0*\hor, 0*\v) circle (.1cm) node (v1) {} node[left] {$x_1$};
        \draw[fill=black] (1*\hor, 0*\v) circle (.1cm) node (u1) {} node[right] {$y_1$};
        \draw[fill=black] (0*\hor, 1*\v) circle (.1cm) node (v2) {} node[left] {$x_2$};
        \draw[fill=black] (1*\hor, 1*\v) circle (.1cm) node (u2) {} node[right] {$y_2$};
        \draw[fill=black] (0*\hor, 2*\v) circle (.1cm) node (v3) {} node[left] {$x_3$};
        \draw[fill=black] (1*\hor, 2*\v) circle (.1cm) node (u3) {} node[right] {$y_3$};
        \draw[fill=black] (0*\hor, 3*\v) circle (.1cm) node (v4) {} node[left] {$x_4$};
        \draw[fill=black] (1*\hor, 3*\v) circle (.1cm) node (u4) {} node[right] {$y_4$};

        \draw (v1) -- (u1);
        \draw (v1) -- (u2);
        \draw (v2) -- (u1);
        \draw (v2) -- (u2);
        \draw (v3) -- (u3);
        \draw (v3) -- (u4);
        \draw (v4) -- (u3);
        \draw (v4) -- (u4);

        \draw (v1) -- (u3);
        \draw (v1) -- (u4);
        \draw (v2) -- (u3);
        \draw (v2) -- (u4);
        %\draw (v3) -- (u2);
      \end{tikzpicture}
      \caption*{$H^{(7)}$: Two 4-cycles connected by 4 connected edges.}\label{3case7}
    \end{subfigure}
         \quad
    \begin{subfigure}[t]{2.5in}
      \centering
      \begin{tikzpicture}
        \def\hor{3cm}
        \def\v{-1cm}

        \draw[fill=black] (0*\hor, 0*\v) circle (.1cm) node (v1) {} node[left] {$x_1$};
        \draw[fill=black] (1*\hor, 0*\v) circle (.1cm) node (u1) {} node[right] {$y_1$};
        \draw[fill=black] (0*\hor, 1*\v) circle (.1cm) node (v2) {} node[left] {$x_2$};
        \draw[fill=black] (1*\hor, 1*\v) circle (.1cm) node (u2) {} node[right] {$y_2$};
        \draw[fill=black] (0*\hor, 2*\v) circle (.1cm) node (v3) {} node[left] {$x_3$};
        \draw[fill=black] (1*\hor, 2*\v) circle (.1cm) node (u3) {} node[right] {$y_3$};
        \draw[fill=black] (0*\hor, 3*\v) circle (.1cm) node (v4) {} node[left] {$x_4$};
        %\draw[fill=black] (1*\hor, 3*\v) circle (.1cm) node (u4) {} node[right] {$y_4$};

        \draw (v1) -- (u1);
        \draw (v1) -- (u2);
        \draw (v2) -- (u1);
        \draw (v2) -- (u2);
        \draw (v3) -- (u3);
        \draw (v3) -- (u2);
        \draw (v4) -- (u3);
        \draw (v4) -- (u2);

     %   \draw (v1) -- (u3);
       % \draw (v1) -- (u4);
       % \draw (v2) -- (u3);
        %\draw (v2) -- (u4);
        %\draw (v3) -- (u2);
      \end{tikzpicture}
      \caption*{$H^{(8)}$: Two 4-cycles which share exactly one vertex.}\label{3case8}
    \end{subfigure} \\
    \caption{Induced subgraphs that are obstructions to satisfying condition $\mathbf{N_2}$}\label{3cases}
  \end{figure}

%N2 obstructions part 1
\begin{lem}\label{N2obstruction1}
   $I_{H^{(1)}}$ and $I_{H^{(8)}}$ do not satisfy condition $\mathbf{N}_2$ and in particular $\beta_{1,4}(I_{H^{(1)}})$ and $ \beta_{1,4}(I_{H^{(8)}}) $ are nonzero.
\end{lem}

\begin{proof} The ideals
 $I_{H^{(1)}}$ and $I_{H^{(8)}}$ are  complete intersections generated by two quadrics. % In particular both have graded Betti tables:
% \begin{center}
%\begin{tabular}{c|ccccc}
%      &$0$&$1$ \\ \hline
%      \text{$2$:}&$2$&\text{-}\\\text{$3$:}&\text{-}&1\\\end{tabular}
%      \end{center}
 	In particular $\beta_{1,4}(I_{H^{(1)}}) = 1$.
 
\end{proof}

%N2 obstructions part 2
\begin{lem}\label{N2obstruction2}
For each $i = 2,\ldots,7$, $I_{H^{(i)}}$ does not satisfy condition $\mathbf{N}_2$ and in particular $\beta_{1 ,4}(I_{H^{(i)}}) \neq 0$.
\end{lem}

\begin{proof}
Let $E^{(i)}$ be the edge set of $H^{(i)}$. Let $e_{ij}$ denote the edge $\{x_i, y_j\}$. Let $\underline{\alpha}$ be the multidegree $(1, 1, 1, 1, 1, 1, 1, 1)$. In each $k[E^{(i)}]$, the monomials in multidegree $\underline{\alpha}$ correspond to perfect matchings in $H^{(i)}$. Note that perfect matchings in $H^{(i)}$ for $i \geq 3$ cannot contain any of the edges $e_{13}, e_{14}, e_{23},$ and $e_{24}$, so the perfect matchings of $H^{(i)}$ are precisely the same as those of $H^{(2)}$. It follows that in each of these graphs, 
\[C_{\underline{\alpha}} = \left\{e_{11}e_{22}e_{33}e_{44},\; e_{11}e_{22}e_{34}e_{43},\; e_{12}e_{21}e_{33}e_{44},\; e_{12}e_{21}e_{34}e_{43}\right\},\]
so $\Gamma(\underline{\alpha})$ has  facets
\[\left\{\{e_{11},e_{22},e_{33},e_{44}\},\; \{e_{11},e_{22},e_{34},e_{43}\},\; \{e_{12},e_{21},e_{33},e_{44}\},\; \{e_{12},e_{21},e_{34},e_{43}\}\right\}.\]

\begin{figure}[H]
\centering

\begin{tikzpicture}
\def\size{2cm}
\def\hor{\size}
\def\v{-1*\size}

\path [fill=red!60, opacity=.5] (0, 0) -- (1*\hor, 1*\v) -- (1*\hor, 2*\v) -- (0, 3*\v) -- (0, 0);
\path [fill=blue!30] (0, 0) -- (1*\hor, 1*\v) -- (2*\hor, 1*\v) -- (3*\hor, 0) -- (0, 0);
\path [fill=green!60] (3*\hor, 0) -- (2*\hor, 1*\v) -- (2*\hor, 2*\v) -- (3*\hor, 3*\v) -- (3*\hor, 0);
\path [fill=yellow!60] (3*\hor, 3*\v) -- (0, 3*\v) -- (1*\hor, 2*\v) -- (2*\hor, 2*\v) -- (3*\hor, 3*\v);
\draw (0, 0) -- (0, 3*\v);
\draw (0, 0) -- (3*\hor, 0);
\draw (0, 0) -- (1*\hor, 1*\v);
\draw (0, 0) -- (2*\hor, 1*\v);
\draw (0, 3*\v) -- (3*\hor, 3*\v);
\draw [dashed] (0, 3*\v) -- (1*\hor, 1*\v);
\draw (0, 3*\v) -- (1*\hor, 2*\v);
\draw (3*\hor, 3*\v) -- (1*\hor, 2*\v);
\draw (3*\hor, 3*\v) -- (2*\hor, 2*\v);
\draw (3*\hor, 3*\v) -- (3*\hor, 0);
\draw (3*\hor, 0) -- (2*\hor, 1*\v);
\draw [dashed] (3*\hor, 0) -- (2*\hor, 2*\v);
\draw (1*\hor, 1*\v) -- (2*\hor, 1*\v);
\draw (1*\hor, 1*\v) -- (1*\hor, 2*\v);
\draw (2*\hor, 1*\v) -- (2*\hor, 2*\v);
\draw (1*\hor, 2*\v) -- (2*\hor, 2*\v);
\draw [dashed] (1*\hor, 1*\v) -- (3*\hor, 0);
\draw (0, 0) -- (1*\hor, 2*\v);
\draw [dashed] (0, 3*\v) -- (2*\hor, 2*\v);
\draw (3*\hor, 3*\v) -- (2*\hor, 1*\v);

%vertices
        \draw[fill=black] (0*\hor, 0*\v) circle (.1cm) node (v1) {} node[left] {$e_{11}$};
                \draw[fill=black] (3*\hor, 0*\v) circle (.1cm) node (v2) {} node[right] { \,$e_{44}$};
        \draw[fill=black] (0*\hor, 3*\v) circle (.1cm) node (v3) {} node[left] {$e_{43}$};
                \draw[fill=black] (3*\hor, 3*\v) circle (.1cm) node (v4) {} node[right] { \,$e_{21}$};

 \draw[fill=black] (1*\hor, 1*\v) circle (.1cm) node (v1) {} node[below right] {$e_{22}$};
                \draw[fill=black] (2*\hor, 1*\v) circle (.1cm) node (v2) {} node[below left] { \,$e_{33}$};
        \draw[fill=black] (1*\hor, 2*\v) circle (.1cm) node (v3) {} node[above right] {$e_{34}$};
                \draw[fill=black] (2*\hor, 2*\v) circle (.1cm) node (v4) {} node[above left] { \,$e_{12}$};

\end{tikzpicture}
\caption{A visualization of $\Gamma(\underline{\alpha})$ in which each shaded tetrahedron represents a  facet.}
\end{figure}

\noindent The geometric realization of the abstract simplicial complex is $\Gamma(\underline{\alpha})$ contracts to a circle.  Thus  $\beta_{1, \underline{\alpha}}(I_{H^{(i)}}) = \dim_k \widetilde{H}_1(\Gamma(\underline{\alpha});k) = 1$.  By Theorem~\ref{InducedGraphBetti}, $\beta_{1,4}(I_{H^{(i)}}) \neq 0$ for $i = 2,\ldots,7$.

\end{proof}

%\begin{lem}\label{N2obstruction3}
%$I_{H^{(8)}}$ does not satisfy $\mathbf{N}_2$.
%\end{lem}

%\begin{proof}
%TODO (It's a complete intersection.)
%\end{proof}

The following graph is the main obstruction to satisfying condition $\mathbf{N}_3$.

\begin{lem}\label{N3obstruction}
Let $H$ be the bipartite graph pictured below.
 \begin{figure}[H]
    \centering

      \begin{tikzpicture}
        \def\hor{3cm}
        \def\v{-1cm}
        
        \draw[fill=black] (0*\hor, 0*\v) circle (.1cm) node (v1) {} node[left] {$x_1$};
        \draw[fill=black] (1*\hor, 0*\v) circle (.1cm) node (u1) {} node[right] {$y_1$};
        \draw[fill=black] (0*\hor, 1*\v) circle (.1cm) node (v2) {} node[left] {$x_2$};
        \draw[fill=black] (1*\hor, 1*\v) circle (.1cm) node (u2) {} node[right] {$y_2$};
        % \draw (0*\hor, 2*\v) node (lell) {\vdots};
        % \draw (1*\hor, 2*\v) node (rell) {\vdots};
        \draw[fill=black] (0*\hor, 2*\v) circle (.1cm) node (v3) {} node[left] {$x_3$};
        \draw[fill=black] (1*\hor, 2*\v) circle (.1cm) node (u3) {} node[right] {$y_3$};

        \draw (v1) -- (u1);
                \draw (v1) -- (u2);

        \draw (v2) -- (u1);
        \draw (v2) -- (u2);
                \draw (v2) -- (u1);
        \draw (v2) -- (u3);

        % \draw (lell) -- (u2);
        % \draw (v3) -- (rell);
        \draw (v3) -- (u1);
        \draw (v1) -- (u3);
        \draw (v3) -- (u2);
      \end{tikzpicture}
 
     \end{figure}
Then $I_H$ is Gorenstein and has graded Betti table:\\
\begin{center}
\begin{tabular}{c|ccc}
      &$0$&$1$&$2$\\\hline
      \textrm{$2\!:$}&$5$&$5$&\text{-}\\\text{$3\!:$}&\text{-}&\text{-}&$1$\\\end{tabular}
      \end{center}
      In particular, $\beta_{2,5}(I_H) \neq 0$ and so $I_H$ does not satisfy condition $N_3$.
%\[\beta_{ij}(I_G) = \begin{cases} 5 & \text{ if } (i,j) = (0,2), (1,3)\\
%1 & \text{ if } (i,j) = (2,5)\\
%0 & \text{ otherwise.} \end{cases}\]
\end{lem}

\begin{proof} %Let $S = K[x_1,x_2,x_3,y_1,y_2,y_3]$ and $R = [e_{ij}\,|\,\{v_i,v_j\} \in E]$.  Define $\phi:R \to S$ by $e_{ij} \mapsto x_iy_j$.  
It is easy to check that $I_H$ is generated by the $4 \times 4$ Pfaffians of the $5 \times 5$ alternating matrix:
\[\mathbf{M} = \begin{pmatrix} 
 0 & e_{13} & e_{23} & e_{31}& e_{32}\\
-e_{13} & 0 & 0 & e_{11} & e_{12}\\
-e_{23} & 0 & 0 & e_{21}& e_{22}\\
-e_{31} & -e_{11} & -e_{21}& 0 & 0\\
-e_{32} & -e_{12}& -e_{22} & 0 & 0
\end{pmatrix}.\]
By \cite[Theorem 2.1]{BE}, it follows that $I_H$ is a Gorenstein, height $3$ ideal.  The claim follows from the symmetry of resolutions of Gorenstein ideals.
\end{proof}

\begin{lem}\label{N4obstruction}
%Let $K_{3,3}$ denote the complete bipartite graph pictured in Figure~\ref{K33}.
 %\begin{figure}[H]
   % \centering

%      \begin{tikzpicture}
        %\def\hor{3cm}
        %\def\v{-1cm}
        
        %\draw[fill=black] (0*\hor, 0*\v) circle (.1cm) node (v1) {} node[left] {$x_1$};
        %\draw[fill=black] (1*\hor, 0*\v) circle (.1cm) node (u1) {} node[right] {$y_1$};
        %\draw[fill=black] (0*\hor, 1*\v) circle (.1cm) node (v2) {} node[left] {$x_2$};
        %\draw[fill=black] (1*\hor, 1*\v) circle (.1cm) node (u2) {} node[right] {$y_2$};
        % \draw (0*\hor, 2*\v) node (lell) {\vdots};
        % \draw (1*\hor, 2*\v) node (rell) {\vdots};
        %\draw[fill=black] (0*\hor, 2*\v) circle (.1cm) node (v3) {} node[left] {$x_3$};
        %\draw[fill=black] (1*\hor, 2*\v) circle (.1cm) node (u3) {} node[right] {$y_3$};

        %\draw (v1) -- (u1);
        %\draw (v1) -- (u2);
        %\draw (v2) -- (u1);
        %\draw (v2) -- (u2);
        %\draw (v2) -- (u1);
        %\draw (v2) -- (u3);

        % \draw (lell) -- (u2);
        % \draw (v3) -- (rell);
        %\draw (v3) -- (u1);
        %\draw (v1) -- (u3);
       % \draw (v3) -- (u2);
     %   \draw (v3) -- (u3);

   %   \end{tikzpicture}
 
%\caption{ The complete bipartite graph $K_{3,3}$ is the main obstruction to satisfying condition $\mathbf{N}_4$.}\label{K33}
 
 %    \end{figure}
The ideal $I_{K_{3,3}}$ is Gorenstein and has graded Betti table:\\
\begin{center}
\begin{tabular}{c|ccccc}
      &$0$&$1$&$2$&$3$&\\ \hline
      \text{$2\!:$}&$9$&$16$&$9$&\text{-}\\\text{$3\!:$}&\text{-}&\text{-}&\text{-}&$1$\\\end{tabular}
      \end{center}
      In particular, $\beta_{3,6}(I_{K_{3,3}}) \neq 0$ and so $I_{K_{3,3}}$ does not satisfy condition $\mathbf{N}_4$.
%\[\beta_{ij}(I_G) = \begin{cases} 5 & \text{ if } (i,j) = (0,2), (1,3)\\
%1 & \text{ if } (i,j) = (2,5)\\
%0 & \text{ otherwise.} \end{cases}\]
\end{lem}

\begin{proof} The ideal $I_{K_{3,3}}$ is generated by the $2 \times 2$ minors of a generic $3 \times 3$ matrix of linear forms and so is Gorenstein and has the above resolution by \cite{GN}.
\end{proof}

%%%%%%%%%%%%%%%%%%%%%%%%%%%%%%%%%%%%%%%%%%%%%%%%%%%%%%%%%%%%%%%%%%%%%%%%%%%%%%%%%%%%%%%%%%%%%%
\section{A Graph-theoretic Result}\label{graph}

This section contains a purely combinatorial characterization of the types of graphs which we show in the following section define linearly presented toric edge ideals.  We show that trees of diameter at most 3 can be characterized locally by the absence of certain induced subgraphs on 4 vertices.  We then show that a bipartite graph with minimum vertex degree at least 2 such that every cycle of size 6 or greater has a chord and such that its bipartite complement is a tree of diameter at most 3 can also be characterized by the absence of 8 particular graphs on at most 8 vertices.

\begin{prop}\label{tree3}
Let $G$ be a graph.  Then $G$ is a tree of diameter at most 3 if and only if every induced subgraph is essentially connected and has no cycles.
\end{prop}

\begin{proof}  Suppose $G$ is a tree of diameter at most $3$.  Since $G$ has no cycles, clearly the same is true for any induced subgraph.  Since the diameter of $G$ is at most $3$, either $G$ has no edges or there exist two vertices $v_1, v_2$ such that every edge of $G$ is incident to $v_1$ or $v_2$.  Let $H$ be an induced subgraph of $G$.  If $H$ contains neither $v_1$ nor $v_2$, then $H$ contains no edges and so is essentially connected.  If $H$ contains exactly one of these vertices, say $v_1$ but not $v_2$, then $H_1$ consists only of edges incident to $v_1$ and isolated vertices, in which case $H$ is also essentially connected.  Finally if $H$ contains both $v_1$ and $v_2$, then all edges of $H$ are incident to  $v_1$ or $v_2$ and so $H$ is essentially connected.

The converse follows easily since $G$ is connected and has no cycles by assumption.
\end{proof}

Note that there is no similar statement for trees of diameter at most $4$.  Indeed, consider a path graph with $4$ edges and $5$ vertices $v_1,\ldots,v_5$.  Then the induced subgraph on vertex set $\{v_1, v_2, v_4, v_5\}$ is not essentially connected.

The following result is our main combinatorial result that allows us to take local obstructions in the form of forbidden induced subgraphs and translate them into a global statement about certain bipartite graphs.

\begin{thm}\label{GraphConditions}
Let $G$ be a bipartite graph with $\delta(G) \geq 2$, such that each cycle of length $\geq 6$ has a chord and $(\overline{G})_1$ is not a tree of diameter at most 3. Then $G$ contains $H^{(i)}$ for some $1 \leq i \leq 8$ as an induced subgraph.
\end{thm}

\begin{proof} It follows from the previous proposition that $(\bar{G})_1$ contains a $4$-cycle or two nonadjacent edges.  Before handling these two cases (Cases 4 and 5 below), we first prove intermediate cases.\\
\noindent\textbf{Case 1:} \emph{$G$ is disconnected.}

Because $G$ has $\delta(G) \geq 2$, each connected component must have a cycle. Since each cycle of length $\geq 6$ has a chord, each connected component must have a 4-cycle. Taking the induced subgraph on two four cycles from distinct connected components will then yield $H^{(2)}$.

\noindent\textbf{Case 2:} \emph{$G$ has a bridge.}

Removing the bridge results in a graph with two connected components, each of which has at most one vertex of degree 1 and all other vertices of degree at least 2. Then each connected component must have a 4-cycle, so the graph contains an induced copy of $H^{(2)}$.

\noindent\textbf{Case 3:} \emph{$G$ has a path with 5 edges as an induced subgraph.}

Denote the induced path by $v_1, v_2, v_3, v_4, v_5, v_6$. We emphasize that this path is induced, so there can not be any edges $\{v_i, v_j\}$ for $\abs{i - j} \geq 2$ in the entire graph $G$. First we assume that there is a second path $v_1, w_1, w_2, v_6$ where the $w_i$ are distinct from the $v_j$. These two paths form an 8-cycle.

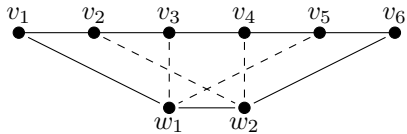
\hfil
\begin{tikzpicture}
\foreach \n in {1, ..., 6}
	\draw[fill=black] (\n, 0) circle (.075cm) node (x\n) {} node[above] {$v_{\n}$};
\foreach \n in {1, ..., 5}
	\draw (\n, 0) -- ({\n + 1}, 0);
\draw[fill=black] (3, -1) circle (.075cm) node (y1) {} node[below] {$w_1$};
\draw[fill=black] (4,-1) circle (.075cm) node (y2) {} node[below] {$w_2$};
\draw (x1) -- (y1) -- (y2) -- (x6);
\draw[dashed] (x2) -- (y2);
\draw[dashed] (x3) -- (y1);
\draw[dashed] (x4) -- (y2);
\draw[dashed] (x5) -- (y1);
\end{tikzpicture}

Because $G$ is bipartite, there are four possible chords in the 8-cycle: $\{v_2, w_2\}, \{v_3, w_1\}, \{v_4, w_2\}, $ and $\{v_5, w_1\}$. If both of the chords $\{v_2, w_2\}$ and $\{v_5, w_1\}$ are present, then the induced subgraph on the vertices $\{v_1, v_2, v_5, v_6, w_1, w_2\}$ is $H^{(1)}$. So we may assume w.l.o.g. that $\{v_2, w_2\}$ is not present. In this case, the chord $\{v_3, w_1\}$ must be present, as must at least one more chord. If the chord $\{v_4, w_2\}$ is present, then the induced subgraph on $\{v_1, v_2, v_3, v_4, w_1, w_2\}$ is $H^{(1)}$. Otherwise, the chord $\{v_5, w_1\}$ is present, in which case the induced subgraph on $\{v_1, v_2, v_3, v_4, v_5, w_1\}$ is $H^{(1)}$. 

Now we may suppose there is no such $w_1$ and $w_2$. Because the edge $\{v_1, v_2\}$ is not a bridge, there is some path from $v_1$ to $v_6$ which avoids it. Consider such a path $w_1, w_2, \hdots, w_m$ of minimal length. The union of the original path with this new path must have cycle containing both of the edges $\{v_1, v_2\}$ and $\{v_1, w_1\}$. If there is an edge $\{v_1, w_i\}$ for any $i > 1$, we contradict the minimality of the new path, so any chord in this cycle must not be incident to $v_1$.  So the edges $\{v_1, v_2\}$ and $\{v_1, w_1\}$ must be contained in the same 4-cycle created by adding chords to the large cycle. This gives two possibilities, either the fourth vertex in the cycle is $v_3$ or $w_2$.  Similarly, we get a path $z_1, \hdots, z_r$ from $v_6$ to $v_1$ avoiding the edge $\{v_5, v_6\}$, which gives two possibilities for cycles $v_6, z_1, z_2, v_5$ or $v_6, z_1, v_4, v_5$.  By symmetry, there are three cases to consider.

\noindent\textbf{Case 3a: } The two $4$-cycles include $v_3$ and $v_4$.
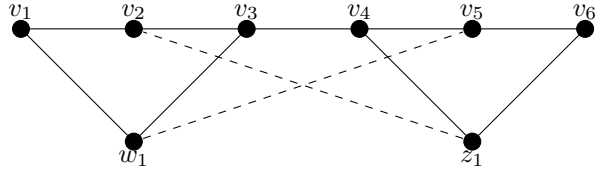
\hfil
\begin{figure}[H]
\begin{tikzpicture}[scale=1.5]
\foreach \n in {1, ..., 6}
	\draw[fill=black] (\n, 0) circle (.075cm) node (x\n) {} node[above] {$v_{\n}$};
\foreach \n in {1, ..., 5}
	\draw (\n, 0) -- ({\n + 1}, 0);
\draw[fill=black] (2, -1) circle (.075cm) node (y1) {} node[below] {$w_1$};
\draw[fill=black] (5, -1) circle (.075cm) node (z1) {} node[below] {$z_1$};
\draw (1,0) -- (2,-1) -- (3,0);
\draw (4,0) -- (5,-1) -- (6,0);
\draw[dashed] (z1) -- (x2);
\draw[dashed] (y1) -- (x5);
\end{tikzpicture}
\end{figure}

The edge $\{w_1, z_1\}$ cannot exist, as we assumed there was no disjoint path of length 3 from $v_1$ to $v_6$; so the only possible additional edges are $\{v_2, z_1\}$ and $\{w_1, v_5\}$. If both of these edges are present, the induced subgraph on $\{v_1, v_2, v_5, v_6, w_1, z_1\}$ is a 6-cycle with no chord, a contradiction, so at least one of the two edges must be missing, which w.l.o.g. we may take to be $\{v_2, z_1\}$. If the edge $\{w_1, v_5\}$ is also missing, then the induced subgraph on all 8 vertices is $H^{(3)}$. If the edge $\{w_1, v_5\}$ is present, then the induced subgraph on $\{v_1, v_2, v_3, v_4, v_5, w_1\}$ is $H^{(1)}$.

%\hfil
%\begin{tikzpicture}
%\draw[fill=black] (0,0) circle (.075cm) node (v1) {} node[left] {$v_1$};
%\draw[fill=black] (2,0) circle (.075cm) node (u1) {} node[right] {$v_2$};
%\draw[fill=black] (0,-1) circle (.075cm) node (v2) {} node[left] {$v_3$};
%\draw[fill=black] (2,-1) circle (.075cm) node (u2) {} node[right] {$w_1$};
%\draw[fill=black] (0,-2) circle (.075cm) node (v3) {} node[left] {$v_5$};
%\draw[fill=black] (2,-2) circle (.075cm) node (u3) {} node[right] {$v_4$};
%\draw[fill=black] (0,-3) circle (.075cm) node (v4) {} node[left] {$z_1$};
%\draw[fill=black] (2,-3) circle (.075cm) node (u4) {} node[right] {$v_6$};
%\draw (0,0) -- (2,0) -- (0,-1) -- (2,-1) -- (0,0);
%\draw (0,-2) -- (2,-2) -- (0,-3) -- (2,-3) -- (0,-2);
%\draw (0,-1) -- (2,-2);
%\end{tikzpicture}

%\vspace{1cm}
%\hrule
%\vspace{1cm}

\noindent\textbf{Case 3b: } The two $4$-cycles include $w_2$ and $v_4$.
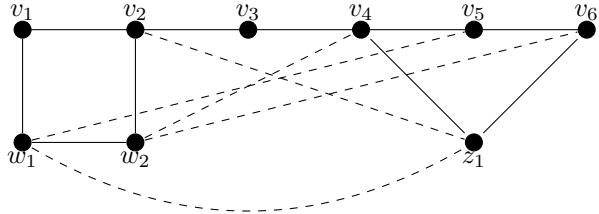
\hfil
\begin{figure}[H]
\begin{tikzpicture}[scale=1.5]
\foreach \n in {1, ..., 6}
	\draw[fill=black] (\n, 0) circle (.075cm) node (x\n) {} node[above] {$v_{\n}$};
\foreach \n in {1, ..., 5}
	\draw (\n, 0) -- ({\n + 1}, 0);
\draw[fill=black] (1, -1) circle (.075cm) node (y1) {} node[below] {$w_1$};
\draw[fill=black] (2, -1) circle (.075cm) node (y2) {} node[below] {$w_2$};
\draw[fill=black] (5, -1) circle (.075cm) node (z1) {} node[below] {$z_1$};
\draw (x1) -- (y1) -- (y2) -- (x2);
\draw (x4) -- (z1) -- (x6);
\draw[dashed] (y1) -- (x5);
\draw[dashed] (y1) edge [bend right] (z1);
\draw[dashed] (y2) -- (x4);
\draw[dashed] (y2) -- (x6);
\draw[dashed] (x2) -- (z1);
\end{tikzpicture}
\end{figure}

First note that the edge $\{w_1, v_3\}$ cannot be present, as then we would be in the previous case. We consider the possible edges between the two 4-cycles: $\{w_1, v_5\}, \{w_1, z_1\}, \{w_2, v_4\}, \{w_2, v_6\},$ and $\{v_2, z_1\}$. 
If the edge $\{w_2, v_4\}$ is present, the induced subgraph on $\{v_1, v_2, v_3, v_4, w_1, w_2\}$ is $H^{(1)}$, so we can assume that it is not present.
Either of the edges $\{w_1, z_1\}$ and $\{w_2, v_6\}$ would give us a disjoint path of length 3, reducing to a previous case.  This leaves us with the following picture:

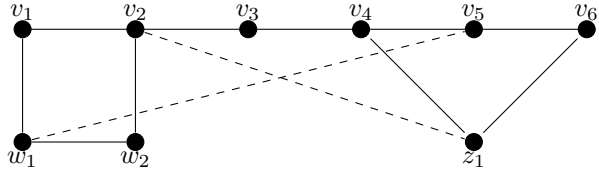
\hfil
\begin{figure}[H]
\begin{tikzpicture}[scale=1.5]
\foreach \n in {1, ..., 6}
	\draw[fill=black] (\n, 0) circle (.075cm) node (x\n) {} node[above] {$v_{\n}$};
\foreach \n in {1, ..., 5}
	\draw (\n, 0) -- ({\n + 1}, 0);
\draw[fill=black] (1, -1) circle (.075cm) node (y1) {} node[below] {$w_1$};
\draw[fill=black] (2, -1) circle (.075cm) node (y2) {} node[below] {$w_2$};
\draw[fill=black] (5, -1) circle (.075cm) node (z1) {} node[below] {$z_1$};
\draw (x1) -- (y1) -- (y2) -- (x2);
\draw (x4) -- (z1) -- (x6);
\draw[dashed] (y1) -- (x5);
%\draw[dashed] (y2) -- (x4);
\draw[dashed] (x2) -- (z1);
\end{tikzpicture}
\end{figure}

\noindent If the edge $\{w_1, v_5\}$ were present, it would produce a $6$-cycle with no chord. So the induced subgraph on the vertices $\{v_1, v_2, v_4, v_5, v_6, w_1, w_2, z_1\}$ is either $H^{(2)}$ or $H^{(3)}$, depending on whether or not the edge $\{v_2, z_1\}$ is present.

%\hfil
%\begin{tikzpicture}
%\draw[fill=black] (0,0) circle (.075cm) node (v1) {} node[left] {$v_1$};
%\draw[fill=black] (2,0) circle (.075cm) node (u1) {} node[right] {$v_2$};
%\draw[fill=black] (0,-1) circle (.075cm) node (v2) {} node[left] {$w_2$};
%\draw[fill=black] (2,-1) circle (.075cm) node (u2) {} node[right] {$w_1$};
%\draw[fill=black] (0,-2) circle (.075cm) node (v3) {} node[left] {$v_5$};
%\draw[fill=black] (2,-2) circle (.075cm) node (u3) {} node[right] {$v_4$};
%\draw[fill=black] (0,-3) circle (.075cm) node (v4) {} node[left] {$z_1$};
%\draw[fill=black] (2,-3) circle (.075cm) node (u4) {} node[right] {$v_6$};
%\draw (v1) -- (u1) -- (v2) -- (u2) -- (v1);
%\draw (v3) -- (u3) -- (v4) -- (u4) -- (v3);
%\draw[dashed] (v4) -- (u1);
%\end{tikzpicture}

%\vspace{1cm}
%\hrule
%\vspace{1cm}

\noindent\textbf{Case 3c: } The two $4$-cycles include $w_2$ and $z_2$.
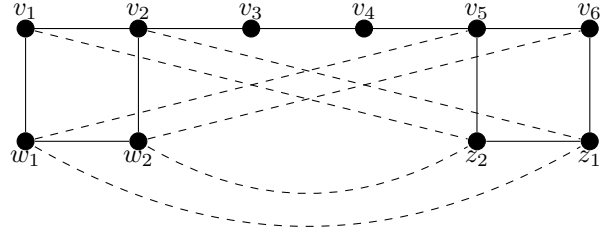
\hfil
\begin{figure}[H]
\begin{tikzpicture}[scale=1.5]
\foreach \n in {1, ..., 6}
	\draw[fill=black] (\n, 0) circle (.075cm) node (x\n) {} node[above] {$v_{\n}$};
\foreach \n in {1, ..., 5}
	\draw (\n, 0) -- ({\n + 1}, 0);
\draw[fill=black] (1, -1) circle (.075cm) node (y1) {} node[below] {$w_1$};
\draw[fill=black] (2, -1) circle (.075cm) node (y2) {} node[below] {$w_2$};
\draw[fill=black] (6, -1) circle (.075cm) node (z1) {} node[below] {$z_1$};
\draw[fill=black] (5, -1) circle (.075cm) node (z2) {} node[below] {$z_2$};
\draw (x1) -- (y1) -- (y2) -- (x2);
\draw (x6) -- (z1) -- (z2) -- (x5);
\draw[dashed] (x1) -- (z2);
\draw[dashed] (y1) -- (x5);
\draw[dashed] (y1) edge[bend right] (z1);
\draw[dashed] (y2) edge[bend right] (z2);
\draw[dashed] (y2) -- (x6);
\draw[dashed] (x2) -- (z1);

\end{tikzpicture}
\end{figure}

In this case, the possible edges between the two cycles are $\{v_1, z_2\}, \{w_1, v_5\}, \{w_1, z_1\}, \{w_2, v_6\}, \{w_2, z_2\},$ and $\{v_2, z_1\}$. 
If any of the edges $\{v_1, z_2\}, \{w_1, z_1\},$ or $\{w_2, v_6\}$ are present, we have a disjoint path of length 3 and are in a previous case. 
If either of $\{w_1, v_5\}$ or $\{v_2, z_1\}$ are present, we would have a $6$-cycle with no chord. So the induced subgraph on $\{v_1, v_2, v_5, v_6, w_1, w_2, z_1, z_2\}$ is either $H^{(2)}$ or $H^{(3)}$, depending on whether $\{w_2, z_2\}$ is present.

%\hfil
%\begin{tikzpicture}
%\draw[fill=black] (0,0) circle (.075cm) node (v1) {} node[left] {$v_1$};
%\draw[fill=black] (2,0) circle (.075cm) node (u1) {} node[right] {$v_2$};
%\draw[fill=black] (0,-1) circle (.075cm) node (v2) {} node[left] {$w_2$};
%\draw[fill=black] (2,-1) circle (.075cm) node (u2) {} node[right] {$w_1$};
%\draw[fill=black] (0,-2) circle (.075cm) node (v3) {} node[left] {$v_5$};
%\draw[fill=black] (2,-2) circle (.075cm) node (u3) {} node[right] {$z_2$};
%\draw[fill=black] (0,-3) circle (.075cm) node (v4) {} node[left] {$z_1$};
%\draw[fill=black] (2,-3) circle (.075cm) node (u4) {} node[right] {$v_6$};
%\draw (v1) -- (u1) -- (v2) -- (u2) -- (v1);
%\draw (v3) -- (u3) -- (v4) -- (u4) -- (v3);
%\draw[dashed] (v2) -- (u3);
%\end{tikzpicture}

%\vspace{1cm}
%\hrule
%\vspace{1cm}

\noindent\textbf{Case 4:} $(\overline{G})_1$ contains a cycle.

Let $x_1, y_1, x_2, y_2$ be the cycle missing from $G$. Since we can assume our graph is connected, there is a shortest path from $x_1$ to $x_2$. If the shortest path has length at least 6, we have an induced path of length 5 and are thus in the previous case. So the shortest path has length 2 or 4. 

\noindent\textbf{Case 4a:} The shortest path between $x_1$ and $x_2$ is $x_1, z_1, z_2, z_3, x_2$.

Since $\delta(G) \geq 2$, $x_1$ must be adjacent to another vertex $z_4$ distinct from $z_1$ and $z_3$, and $x_2$ must be adjacent to another vertex $z_5$ distinct from $z_1$ and $z_3$. If $z_4 = z_5$, we have a shorter path from $x_1$ to $x_2$, so $z_4$ and $z_5$ must be distinct. 

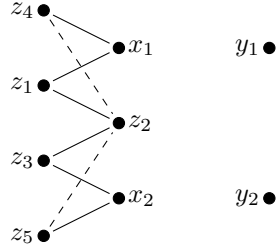
\hfil
\begin{tikzpicture}
\draw[fill=black] (0,0) circle (.075cm) node (x1) {} node[right] {$x_1$};
\draw[fill=black] (2,0) circle (.075cm) node (y1) {} node[left] {$y_1$};
\draw[fill=black] (0,-2) circle (.075cm) node (x2) {} node[right] {$x_2$};
\draw[fill=black] (2,-2) circle (.075cm) node (y2) {} node[left] {$y_2$};
\draw[fill=black] (0,-1) circle (.075cm) node (z2) {} node[right] {$z_2$};
\draw[fill=black] (-1,-.5) circle (.075cm) node (z1) {} node[left] {$z_1$};
\draw[fill=black] (-1,-1.5) circle (.075cm) node (z3) {} node[left] {$z_3$};
\draw[fill=black] (-1, .5) circle (.075cm) node (z4) {} node[left] {$z_4$};
\draw[fill=black] (-1, -2.5) circle (.075cm) node (z5) {} node[left] {$z_5$};
\draw (z4) -- (x1) -- (z1) -- (z2) -- (z3) -- (x2) -- (z5);
\draw[dashed] (z4) -- (z2);
\draw[dashed] (z5) -- (z2);
\end{tikzpicture}

If the edges $\{x_2, z_1\}$ or $\{x_1, z_3\}$ are present, we have a shorter path, so these edges must be missing. If either of the edges $\{z_4, z_2\}$ or $\{z_5, z_2\}$ are missing, we have an induced path of length 5, putting us in the previous case. So the induced subgraph on $\{x_1, x_2, z_1, z_2, z_3, z_4, z_5\}$ is $H^{(8)}$.

\noindent\textbf{Case 4b:} The shortest path between $x_1$ and $x_2$ is $x_1, z_1, x_2$. 

Since $\delta(G) \geq 2$, $x_1$ must be adjacent to another vertex $z_2$, and $x_2$ must be adjacent to another vertex $z_3$. 

\noindent\textbf{Case 4b(i):} $z_1$ is the only common neighbor of $x_1$ and $x_2$. \\
If $z_2 \neq z_3$, $z_2$ must be adjacent to another vertex $z_4$. If $z_4 = x_2$, it is a common neighbor of both vertices. Otherwise, consider the induced subgraph on the vertices $\{x_1, x_2, z_1, z_2, z_3, z_4\}$. By assumption, the edges $\{z_3, x_1\}$ and $\{z_2, x_2\}$ cannot be present, so the only possibilities are $\{z_4, z_1\}$ and $\{z_4, z_3\}$. 

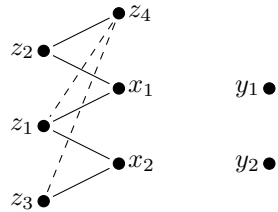
\hfil
\begin{tikzpicture}
\draw[fill=black] (0,0) circle (.075cm) node (x1) {} node[right] {$x_1$};
\draw[fill=black] (2,0) circle (.075cm) node (y1) {} node[left] {$y_1$};
\draw[fill=black] (0,-1) circle (.075cm) node (x2) {} node[right] {$x_2$};
\draw[fill=black] (2,-1) circle (.075cm) node (y2) {} node[left] {$y_2$};
\draw[fill=black] (-1,.5) circle (.075cm) node (z2) {} node[left] {$z_2$};
\draw[fill=black] (-1,-.5) circle (.075cm) node (z1) {} node[left] {$z_1$};
\draw[fill=black] (-1,-1.5) circle (.075cm) node (z3) {} node[left] {$z_3$};
\draw[fill=black] (0, 1) circle (.075cm) node (z4) {} node[right] {$z_4$};
%\draw[fill=black] (-1, -2.5) circle (.075cm) node (z5) {} node[left] {$z_5$};
\draw (z4) -- (z2) -- (x1) -- (z1) -- (x2) -- (z3);
\draw[dashed] (z4) -- (z1);
\draw[dashed] (z4) -- (z3);
\end{tikzpicture}

If neither of these are present, we have an induced path of length 5, which is case 3. If both edges are present, then the induced subgraph is $H^{(1)}$. If only $\{z_4, z_3\}$ is present, we have a cycle of length 6 with no chord, a contradiction. 
If only $\{z_4, z_1\}$ is present, we can find a second neighbor of $z_3$, which we label $z_5$. We apply the same analysis to the induced subgraph on $\{x_1, x_2, z_1, z_2, z_3, z_5\}$ and the only case we have not already argued is if the only additional edge is $\{z_1, z_5\}$. In this case, the induced subgraph on the seven vertices $\{x_1, x_2, z_1, z_2, z_3, z_4, z_5\}$ is $H^{(8)}$. 

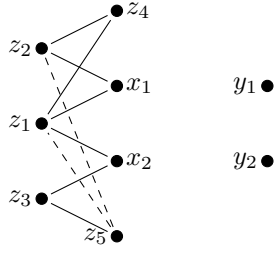
\hfil
\begin{tikzpicture}
\draw[fill=black] (0,0) circle (.075cm) node (x1) {} node[right] {$x_1$};
\draw[fill=black] (2,0) circle (.075cm) node (y1) {} node[left] {$y_1$};
\draw[fill=black] (0,-1) circle (.075cm) node (x2) {} node[right] {$x_2$};
\draw[fill=black] (2,-1) circle (.075cm) node (y2) {} node[left] {$y_2$};
\draw[fill=black] (-1,.5) circle (.075cm) node (z2) {} node[left] {$z_2$};
\draw[fill=black] (-1,-.5) circle (.075cm) node (z1) {} node[left] {$z_1$};
\draw[fill=black] (-1,-1.5) circle (.075cm) node (z3) {} node[left] {$z_3$};
\draw[fill=black] (0, 1) circle (.075cm) node (z4) {} node[right] {$z_4$};
\draw[fill=black] (0, -2) circle (.075cm) node (z5) {} node[left] {$z_5$};
%\draw[fill=black] (-1, -2.5) circle (.075cm) node (z5) {} node[left] {$z_5$};
\draw (z4) -- (z2) -- (x1) -- (z1) -- (x2) -- (z3) -- (z5);
\draw (z4) -- (z1);
\draw[dashed] (z5) -- (z1);
\draw[dashed] (z5) -- (z2);
\end{tikzpicture}

\noindent\textbf{Case 4b(ii):} $z_2$ and $z_3$ can be chosen to be the same.

In this case, we have a 4-cycle $x_1, z_1, x_2, z_2$. We then consider the shortest path from $y_1$ to $y_2$ and apply all prior case 4 analysis to this path. The only case we have not then argued is if there is also a 4-cycle $y_1, w_1, y_2, w_2$. In this case, the induced subgraph on the vertices $\{x_1, x_2, z_1, z_2, y_1, y_2, w_1, w_2\}$ is one of $H^{(2)}, H^{(3)}, H^{(4)}, H^{(5)}, H^{(6)},$ or $H^{(7)}$, depending on what edges are present between the $z_i$ and $w_j$. 

\noindent\textbf{Case 5:} $(\overline{G})_1$ has two nonadjacent edges.

In this case, $G$ contains 4 vertices $x_1, x_2, y_1, y_2$ such that $\{x_1, y_2\}$ and $\{x_2, y_1\}$ are edges, while $\{x_1, y_1\}$ and $\{x_2, y_2\}$ are non-edges. Because $\delta(G) \geq 2$, $x_2$ is adjacent to another vertex, $y_3$, and $y_2$ is adjacent to another vertex $x_3$. The edges $\{x_1, y_3\}, \{x_3, y_1\},$ and $\{x_3, y_3\}$ may or may not be present.

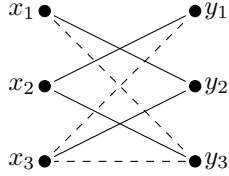
\hfil
\begin{tikzpicture}
\draw[fill=black] (0,0) circle (.075cm) node (x1) {} node[left] {$x_1$};
\draw[fill=black] (2,0) circle (.075cm) node (y1) {} node[right] {$y_1$};
\draw[fill=black] (0,-1) circle (.075cm) node (x2) {} node[left] {$x_2$};
\draw[fill=black] (2,-1) circle (.075cm) node (y2) {} node[right] {$y_2$};
\draw[fill=black] (0,-2) circle (.075cm) node (x3) {} node[left] {$x_3$};
\draw[fill=black] (2,-2) circle (.075cm) node (y3) {} node[right] {$y_3$};
\draw (x1) -- (y2) -- (x3);
\draw (y1) -- (x2) -- (y3);
\draw[dashed] (x1) -- (y3) -- (x3) -- (y1);
\end{tikzpicture}

\noindent\textbf{Case 5a:} None of these three edges are present.

In this case, $x_1, y_1, x_3, y_3$ gives a 4-cycle in $(\overline{G})_1$ which is case 4. 

\noindent\textbf{Case 5b:} All three edges are present.

In this case, our graph is $H^{(1)}$. 

\noindent\textbf{Case 5c:} Exactly one of $\{x_1, y_3\}$, $\{x_3, y_1\}$, and $\{x_3,y_3\}$ are present.

In this case, the induced subgraph on all six vertices is an induced path of length 5, which is case 3.

\noindent\textbf{Case 5d:} Only the edge $\{x_3, y_3\}$ is missing.

In this case, $G$ has a 6-cycle with no chord, a contradiction.

\noindent\textbf{Case 5e:} Only the edge $\{x_3, y_1\}$ is missing.

In this case, $y_1$ must be adjacent to another vertex $x_4$ and the edges $\{x_4, y_2\}$ and $\{x_4, y_3\}$ may or may not be present.

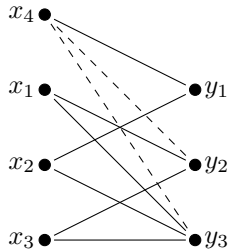
\hfil
\begin{tikzpicture}
\draw[fill=black] (0,0) circle (.075cm) node (x1) {} node[left] {$x_1$};
\draw[fill=black] (2,0) circle (.075cm) node (y1) {} node[right] {$y_1$};
\draw[fill=black] (0,-1) circle (.075cm) node (x2) {} node[left] {$x_2$};
\draw[fill=black] (2,-1) circle (.075cm) node (y2) {} node[right] {$y_2$};
\draw[fill=black] (0,-2) circle (.075cm) node (x3) {} node[left] {$x_3$};
\draw[fill=black] (2,-2) circle (.075cm) node (y3) {} node[right] {$y_3$};
\draw[fill=black] (0,1) circle (.075cm) node (x4) {} node[left] {$x_4$};
\draw (x1) -- (y2) -- (x3);
\draw (y1) -- (x2) -- (y3);
\draw (x1) -- (y3) -- (x3);
\draw (y1) -- (x4);
\draw[dashed] (y2) -- (x4) -- (y3);
\end{tikzpicture}

\noindent \textbf{Case 5e(i):} Both $\{x_4, y_2\}$ and $\{x_4, y_3\}$ are present. 

In this case, the induced subgraph on all seven vertices is $H^{(8)}$. 

\noindent\textbf{Case 5e(ii): } Only $\{x_4, y_2\}$ is present. 

In this case, the induced subgraph on $\{x_1, x_2, x_4, y_1, y_2, y_3\}$ is a 6-cycle with no chord, a contradiction.

\noindent\textbf{Case 5e(iii):} Only $\{x_4, y_3\}$ is present.

In this case, the induced subgraph on all seven vertices is $H^{(8)}$.

\noindent\textbf{Case 5e(iv):} Both edges are missing.

In this case, the induced subgraph on $\{x_4, y_1, x_2, y_3, x_3, y_2\}$ is a path of length 5, which is case 3. 

\noindent\textbf{Case 5f:} Only the edge $\{x_1, y_3\}$ is missing. 

This case is identical to case 5e.

%This completes the proof.

%\textbf{Case 5g:} Both $\{v_3, w_1\}$ and $\{v_3, w_3\}$ are missing.

%In this case, the induced subgraph on all six vertices is a path of length 5, which is case 3.

%\textbf{Case 5h:} Both $\{x_1, y_3\}$ and $\{x_3, y_3\}$ are missing.

%Same as previous case.

%To complete the proof, we note that by Proposition~\ref{tree3}, $(\overline{G})_1$ has an induced subgraph which is either not essentially connected, or has a cycle. If the induced subgraph has a 4-cycle, we are in case 4. If the induced subgraph has no 4-cycle, it is either not essentially connected or has a cycle of length greater than 4, so case 5 applies.

\end{proof}

%%%%%%%%%%%%%%%%%%%%%%%%%%%%%%%%%%%%%%%%%%%%%%%%%%%%%%%%%%%%%%%%%%%%%%%%%%%%%%

%%%%%%%%%%%%%%%%%%%%%%%%%%%%%%%%%%%%%%%%%%%%%%%%%%%%%%%%%%%%%%%%%%%%%%%%%%%%%%%%%%%%%%%%%%%%%%
\section{Main Results}\label{main}

In this section we collect the proofs of our main results classifying the graphs whose toric edge ideals satisfy each of the Green-Lazarsfeld conditions $\mathbf{N}_p$.  The previous sections identified certain obstructions in the form of forbidden subgraphs to a given ideal $I_G$ satisfying conditions $\mathbf{N}_2$, $\mathbf{N}_3$, or $\mathbf{N}_4$.  First we use the following result which shows we may focus our attention on the existence of minimal Koszul syzygies of toric edge ideals.

\begin{prop}[cf. {\cite[Proposition 2.8]{M}}]\label{KoszulSchreyer}
If R = S/I is a Koszul algebra, then the first syzygies of $I$ are  minimally generated by linear syzygies and Koszul syzygies.
\end{prop}
We note that it is also possible to prove this result when $I$ has a quadratic Gr\"obner bases (as happens in our case of interest) using Schreyer's Theorem on syzygies.  See \cite[Theorem 3.3]{CLO}.  A similar observation was made in \cite[Theorem 3.1]{EHH}.

It would be possible to prove the following characterization of linearly presented bipartite toric edge ideals by appealing to Theorem~\ref{InducedGraphBetti} and enumerating all possible subgraphs with at most 8 vertices.  While this strategy is useful to enumerate the obstructions to being linearly presented, it is inefficient to check all such graphs by hand.  Instead, by using  the previous proposition, we need only consider induced subgraphs where a potential Koszul syzygy exists and show that it is not a minimal generator of the syzygy module of $I_G$ in every possible case.  This brings us to our first main result.

\begin{thm}\label{N2entire}
  Let $G$ be a bipartite graph with $\delta(G) \geq 2$. Then $I_G$ satisfies $\mathbf{N}_2$ if and only if  $\overline{G}$ is essentially a tree of diameter at most 3.
\end{thm}

\begin{proof}
  \resetcase{} Suppose $\overline{G}$ is not essentially a tree of diameter at most $3$.  If there is a  cycle of length $\ge 6$  that has no chord, $I_G$ is not quadratically generated by Theorem~\ref{OHthm} and thus must fail condition $\mathbf{N}_1$ and also condition $\mathbf{N}_2$.  So we may suppose that all chords of $G$ of length $\ge 6$ have a chord. Now by Theorem~\ref{GraphConditions}, $G$ contains the graph $H = H^{(i)}$ for some $1 \le i \le 8$ as an induced subgraph.  By Lemmas~\ref{N2obstruction1} and \ref{N2obstruction2}, $\beta_{2,4}(I_{H}) \neq 0$.  By Theorem~\ref{InducedGraphBetti}, $\beta_{2,4}(I_G) \neq 0$ and thus $I_G$  does not satisfy property $\mathbf{N}_2$.
  
  Now suppose that $\overline{G}$ is essentially a tree of diameter at most $3$.  By Theorem~\ref{OHthm}, $I_G$ is generated by a Gr\"obner basis of quadrics corresponding to $4$-cycles in $G$ and $S/I_G$ is Koszul. If $H$ is any subgraph of $G$, then any 4-cycle in $H$ is a 4-cycle in $G$, so $(I_H)_2 \subseteq I_G$. Any syzygy of $I_H$ can be extended to a syzygy of $I_G$, so, by Proposition~\ref{KoszulSchreyer}, it is sufficient to show that any Koszul syzygy is a linear combination of linear syzygies.

  Koszul Syzygies correspond to pairs of distinct $4$-cycles. There are five
  possibilities for the configuration of pairs of distinct cycles:
  \begin{enumerate}
  \item they share two edges,
  \item they share an edge,
  \item they share two vertices, but no edges.
  \item they share a vertex but no edges,
  \item or they don't intersect.
  \end{enumerate}

  \case{} $G$ contains a subgraph $H$ which has distinct $4$-cycles sharing exactly two edges.\\
  In this case, the only subgraph satisfying our assumption is $K_{2,3}$; see Figure~\ref{23}.  $I_H$ is then resolved by the linear Eagon-Northcott complex and thus has no minimal quadratic syzygies.

  \begin{figure}[H]
    \centering
    \begin{tikzpicture}
      \def\hor{3cm}
      \def\v{-1cm}

      \draw[fill=black] (0*\hor, 0*\v) circle (.1cm) node (v1) {} node[left] {$x_1$};
      \draw[fill=black] (1*\hor, 0*\v) circle (.1cm) node (u1) {} node[right] {$y_1$};
      \draw[fill=black] (0*\hor, 1*\v) circle (.1cm) node (v2) {} node[left] {$x_2$};
      \draw[fill=black] (1*\hor, 1*\v) circle (.1cm) node (u2) {} node[right] {$y_2$};
      \draw[fill=black] (1*\hor, 2*\v) circle (.1cm) node (u3) {} node[right] {$y_3$};

      \draw (v1) -- (u1);
      \draw (v1) -- (u2);
      \draw (v1) -- (u3);
      \draw (v2) -- (u1);
      \draw (v2) -- (u2);
      \draw (v2) -- (u3);
    \end{tikzpicture}
    \caption{}\label{23}
  \end{figure}

  %where the two cycles are $x_1, y_1, x_2, y_2$ and $x_1, y_2, x_2, y_3$.

  %Calling this graph $H$ and labeling the edge from $x_i$ to $y_j$ by $e_{ij}$, we know

%  \begin{align*}
   % I_H = (&e_{12}e_{21}-e_{11}e_{22},
      %     e_{13}e_{21}-e_{11}e_{23},
         %  e_{13}e_{22}-e_{12}e_{23}).
  %\end{align*}

%  The Koszul syzygy in this case is

%  \[\mat{-e_{13}e_{22} + e_{12}e_{23} \\ 0 \\ e_{12}e_{21} - e_{11}e_{22}} = e_{12}\mat{e_{23} \\ -e_{22} \\ e_{21}} - e_{22} \mat{e_{13} \\ -e_{12} \\ e_{11}}\]

  %The reader can verify that the terms on the right-hand side are linear syzygies.
  
  \case{} $G$ contains a subgraph $H$ which has distinct $4$-cycles sharing exactly one edge.\\
  In this case, the minimal graph containing the two given 4-cycles is pictured in Figure~\ref{case2a},

  \begin{figure}[h]
    \centering
    \begin{subfigure}[t]{2in}
      \centering
      \begin{tikzpicture}
        \def\hor{3cm}
        \def\v{-1cm}

        \draw[fill=black] (0*\hor, 0*\v) circle (.1cm) node (v1) {} node[left] {$x_1$};
        \draw[fill=black] (1*\hor, 0*\v) circle (.1cm) node (u1) {} node[right] {$y_1$};
        \draw[fill=black] (0*\hor, 1*\v) circle (.1cm) node (v2) {} node[left] {$x_2$};
        \draw[fill=black] (1*\hor, 1*\v) circle (.1cm) node (u2) {} node[right] {$y_2$};
        \draw[fill=black] (0*\hor, 2*\v) circle (.1cm) node (v3) {} node[left] {$x_3$};
        \draw[fill=black] (1*\hor, 2*\v) circle (.1cm) node (u3) {} node[right] {$y_3$};

        \draw[blue, dashed, line width=0.3mm] (v1) -- (u3);
        \draw[red, dashed, line width=0.3mm] (v3) -- (u1);
        \draw (v1) -- (u1);
        \draw (v1) -- (u2);
        % \draw (v1) -- (u3);
        \draw (v2) -- (u1);
        \draw (v2) -- (u2);
        \draw (v2) -- (u3);
        \draw (v3) -- (u2);
        \draw (v3) -- (u3);
      \end{tikzpicture}
      \caption{The minimal graph containing the two cycles.}\label{case2a}
    \end{subfigure}
    \hspace{1cm}
    \begin{subfigure}[t]{2in}
      \centering
      \begin{tikzpicture}
        \def\hor{3cm}
        \def\v{-1cm}
        
        \draw[fill=black] (0*\hor + 0*\hor, 0*\v) circle (.1cm) node (nv1) {} node[left] {$x_1$};
        \draw[fill=black] (0*\hor + 1*\hor, 0*\v) circle (.1cm) node (nu1) {} node[right] {$y_1$};
        \draw[fill=black] (0*\hor + 0*\hor, 1*\v) circle (.1cm) node (nv2) {} node[left] {$x_2$};
        \draw[fill=black] (0*\hor + 1*\hor, 1*\v) circle (.1cm) node (nu2) {} node[right] {$y_2$};
        \draw[fill=black] (0*\hor + 0*\hor, 2*\v) circle (.1cm) node (nv3) {} node[left] {$x_3$};
        \draw[fill=black] (0*\hor + 1*\hor, 2*\v) circle (.1cm) node (nu3) {} node[right] {$y_3$};

        % \draw[red, dashed] (v1) -- (u3);
        % \draw[red, dashed] (v3) -- (u1);
        \draw (nv1) -- (nu1);
        \draw (nv1) -- (nu2);
        \draw (nv1) -- (nu3);
        \draw (nv2) -- (nu1);
        \draw (nv2) -- (nu2);
        \draw (nv2) -- (nu3);
        \draw (nv3) -- (nu2);
        \draw (nv3) -- (nu3);
      \end{tikzpicture}
      \caption{Edge $\{x_1, y_3\}$ added.}\label{case2b}
    \end{subfigure}
    \caption{}\label{case2}
  \end{figure}
  \noindent where the two 4-cycles are $x_1, y_1, x_2, y_2$ and $x_2, y_2, x_3, y_3$ and the bipartite complement consists of the two dashed lines. Since $H^{(1)}$ is not an induced subgraph of $G$ by Theorem~\ref{GraphConditions}, at least one of the two missing edges must be present. By symmetry, we can assume that this present edge is the edge from $x_1$ to $y_3$; see Figure~\ref{case2b}.  Calling this graph $H$ and labeling the edge from $x_i$ to $y_j$ by $e_{ij}$, we have
  \begin{align*}
    I_H = (&e_{12}e_{21}-e_{11}e_{22},
           e_{13}e_{21}-e_{11}e_{23},
           e_{13}e_{22}-e_{12}e_{23},
           e_{13}e_{32}-e_{12}e_{33},
           e_{23}e_{32}-e_{22}e_{33}).
  \end{align*}

  \noindent The Koszul syzygy in this case is

  \[\mat{-e_{22}e_{33} + e_{23}e_{32} \\ 0 \\ 0 \\ 0 \\ e_{11}e_{22}-e_{12}e_{21}} = e_{32}\mat{e_{23} \\ -e_{22} \\ e_{21} \\ 0 \\ 0} - e_{21} \mat{0 \\ 0 \\ e_{32} \\ -e_{22} \\ e_{12}} + e_{22}\mat{-e_{33} \\ e_{32} \\ 0 \\ -e_{21} \\ e_{11}}.\]

  \noindent The reader can verify that the terms on the right-hand side are linear syzygies.

  \case{}$G$ contains a subgraph $H$ which has distinct $4$-cycles sharing exactly two vertices but no edges.\\
   In this case, the only subgraph satisfying our assumptions is

  \begin{figure}[H]
    \centering
    \begin{tikzpicture}
      \def\hor{3cm}
      \def\v{-1cm}

      \draw[fill=black] (0*\hor, 0*\v) circle (.1cm) node (v1) {} node[left] {$x_1$};
      \draw[fill=black] (1*\hor, 0*\v) circle (.1cm) node (u1) {} node[right] {$y_1$};
      \draw[fill=black] (0*\hor, 1*\v) circle (.1cm) node (v2) {} node[left] {$x_2$};
      \draw[fill=black] (1*\hor, 1*\v) circle (.1cm) node (u2) {} node[right] {$y_2$};
      \draw[fill=black] (1*\hor, 2*\v) circle (.1cm) node (u3) {} node[right] {$y_3$};
      \draw[fill=black] (1*\hor, 3*\v) circle (.1cm) node (u4) {} node[right] {$y_4$};

      \draw (v1) -- (u1);
      \draw (v1) -- (u2);
      \draw (v2) -- (u1);
      \draw (v2) -- (u2);
      \draw (v1) -- (u3);
      \draw (v1) -- (u4);
      \draw (v2) -- (u3);
      \draw (v2) -- (u4);
    \end{tikzpicture}
    \caption{}
  \end{figure}

 \noindent where the two cycles are $x_1, y_1, x_2, y_2$ and $x_3, y_1, x_4, y_2$.  Once again this is $K_{2,4}$ and $I_{K_{2,4}}$ is resolved by a linear Eagon-Northcott resolution and so, as in Case 1, there are no minimal quadratic syzygies.

 % Calling this graph $H$ and labeling the edge from $x_i$ to $y_j$ by $e_{ij}$, we know

%  \begin{align*}
 %   I_H = (&e_{12}e_{21}-e_{11}e_{22},
%           e_{13}e_{21}-e_{11}e_{23},
 %          e_{14}e_{21}-e_{11}e_{24},\\
  %         &e_{13}e_{22}-e_{12}e_{23},
  %         e_{14}e_{22}-e_{12}e_{24},
  %         e_{14}e_{23}-e_{13}e_{24}).
 % \end{align*}

%  The Koszul syzygy in this case is

%  \[\mat{-{e}_{14} {e}_{23}+{e}_{13} {e}_{24}\\
 %     0\\
 %     0\\
 %     0\\
 %     0\\
 %     {e}_{12} {e}_{21}-{e}_{11} {e}_{22}}
 %   = -e_{11}\mat{0\\
 %     0\\
%      0\\
 %     {e}_{24}\\
  %    {-{e}_{23}}\\
  %    {e}_{22}}
  %  + e_{12}\mat{0\\
   %   {e}_{24}\\
    %  {-{e}_{23}}\\
     % 0\\
     % 0\\
     % {e}_{21}}
   % + e_{24}\mat{{e}_{13}\\
    %  {-{e}_{12}}\\
     % 0\\
     % {e}_{11}\\
     % 0\\
     % 0}
    %- e_{23}\mat{{e}_{14}\\
     % 0\\
      %{-{e}_{12}}\\
      %0\\
      %{e}_{11}\\
      %0}.\]

%  The reader can verify that the terms on the right-hand side are linear syzygies.

  \case{} $G$ contains a subgraph $H$ which has distinct $4$-cycles sharing exactly one vertex and no edges.\\
  In this case, the minimal subgraph containing the two cycles is pictured in Figure~\ref{2case4a},

  \begin{figure}[h]
    \centering
    \begin{subfigure}[t]{2.5in}
      \centering
      \begin{tikzpicture}
        \def\hor{3cm}
        \def\v{-1cm}

        \draw[fill=black] (0*\hor, 0*\v) circle (.1cm) node (v1) {} node[left] {$x_1$};
        \draw[fill=black] (1*\hor, 0*\v) circle (.1cm) node (u1) {} node[right] {$y_1$};
        \draw[fill=black] (0*\hor, 1*\v) circle (.1cm) node (v2) {} node[left] {$x_2$};
        \draw[fill=black] (1*\hor, 1*\v) circle (.1cm) node (u2) {} node[right] {$y_2$};
        \draw[fill=black] (0*\hor, 2*\v) circle (.1cm) node (v3) {} node[left] {$x_3$};
        \draw[fill=black] (1*\hor, 2*\v) circle (.1cm) node (u3) {} node[right] {$y_3$};
        \draw[fill=black] (0*\hor, 3*\v) circle (.1cm) node (v4) {} node[left] {$x_4$};

        \draw (v1) -- (u1);
        \draw (v1) -- (u2);
        \draw (v2) -- (u1);
        \draw (v2) -- (u2);
        \draw (v3) -- (u2);
        \draw (v3) -- (u3);
        \draw (v4) -- (u2);
        \draw (v4) -- (u3);
        \draw[dashed, line width=0.3mm, red] (v1) -- (u3);
        \draw[dashed, line width=0.3mm, red] (v2) -- (u3);
        \draw[dashed, line width=0.3mm, blue] (v3) -- (u1);
        \draw[dashed, line width=0.3mm, blue] (v4) -- (u1);
      \end{tikzpicture}
      \caption{Minimal subgraph containing the cycles.}\label{2case4a}
    \end{subfigure}
    \quad
    \begin{subfigure}[t]{2.5in}
      \centering
      \begin{tikzpicture}
        \def\hor{3cm}
        \def\v{-1cm}

        \draw[fill=black] (0*\hor, 0*\v) circle (.1cm) node (v1) {} node[left] {$x_1$};
        \draw[fill=black] (1*\hor, 0*\v) circle (.1cm) node (u1) {} node[right] {$y_1$};
        \draw[fill=black] (0*\hor, 1*\v) circle (.1cm) node (v2) {} node[left] {$x_2$};
        \draw[fill=black] (1*\hor, 1*\v) circle (.1cm) node (u2) {} node[right] {$y_2$};
        \draw[fill=black] (0*\hor, 2*\v) circle (.1cm) node (v3) {} node[left] {$x_3$};
        \draw[fill=black] (1*\hor, 2*\v) circle (.1cm) node (u3) {} node[right] {$y_3$};
        \draw[fill=black] (0*\hor, 3*\v) circle (.1cm) node (v4) {} node[left] {$x_4$};

        \draw (v1) -- (u1);
        \draw (v1) -- (u2);
        \draw (v2) -- (u1);
        \draw (v2) -- (u2);
        \draw (v3) -- (u2);
        \draw (v3) -- (u3);
        \draw (v4) -- (u2);
        \draw (v4) -- (u3);
        \draw (v1) -- (u3);
        \draw (v2) -- (u3);
      \end{tikzpicture}
      \caption{Minimal satisfactory subgraph.}\label{2case4b}
    \end{subfigure}
    \caption{}
  \end{figure}
  \noindent where the two cycles are $x_1, y_1, x_2, y_2$ and $x_3, y_2, x_4, y_3$. To ensure that $\overline{G}$ is essentially connected, one of the two connected components of $\overline{G}$ (the dashed lines) must be present in $G$. By symmetry, we can assume that the edges $\{x_1, y_3\}$ and $\{x_2, y_3\}$ are present in $G$. The resulting graph is shown in Figure~\ref{2case4b}.

  Calling this graph $H$ and labeling the edge from $x_i$ to $y_j$ by $e_{ij}$, one computes that

  \begin{align*}
    I_H =  (&e_{12}e_{21}-e_{11}e_{22},
            e_{13}e_{21}-e_{11}e_{23},
            e_{13}e_{22}-e_{12}e_{23},
            e_{13}e_{32}-e_{12}e_{33},\\
            &e_{23}e_{32}-e_{22}e_{33},
            e_{13}e_{42}-e_{12}e_{43},
            e_{23}e_{42}-e_{22}e_{43},
            e_{33}e_{42}-e_{32}e_{43}).
  \end{align*}

  \noindent The Koszul syzygy in question is then a sum of linear syzygies, as shown below.  
  
  \[\mat{-e_{32}e_{43} + e_{33}e_{42} \\ 0 \\ 0 \\ 0 \\ 0 \\ 0 \\ 0 \\ e_{11}e_{22}-e_{12}e_{21}} = e_{11}\mat{0 \\ 0 \\ 0 \\ 0 \\ 0 \\ e_{42} \\ -e_{32} \\ e_{22}} - e_{21}\mat{0\\0\\0\\e_{42}\\-e_{32}\\0\\0\\e_{12}} -e_{42}\mat{-e_{33} \\ e_{32} \\ 0 \\ -e_{21} \\ 0 \\ e_{11} \\ 0 \\ 0 } + e_{32}\mat{-e_{43} \\ e_{42} \\ 0 \\ 0 \\ -e_{21} \\ 0 \\ e_{11} \\ 0 }.\]

 % The reader can verify that the terms on the right-hand side are linear syzygies.

  \case{}$G$ contains a subgraph $H$ which has two disjoint $4$-cycles.\\
   In this case, the minimal subgraph of the cycles is pictured in \ref{2case5a},

  \begin{figure}[h]
    \begin{subfigure}[t]{2.5in}
      \centering
      \begin{tikzpicture}
        \def\hor{3cm}
        \def\v{-1cm}

        \draw[fill=black] (0*\hor, 0*\v) circle (.1cm) node (v1) {} node[left] {$x_1$};
        \draw[fill=black] (1*\hor, 0*\v) circle (.1cm) node (u1) {} node[right] {$y_1$};
        \draw[fill=black] (0*\hor, 1*\v) circle (.1cm) node (v2) {} node[left] {$x_2$};
        \draw[fill=black] (1*\hor, 1*\v) circle (.1cm) node (u2) {} node[right] {$y_2$};
        \draw[fill=black] (0*\hor, 2*\v) circle (.1cm) node (v3) {} node[left] {$x_3$};
        \draw[fill=black] (1*\hor, 2*\v) circle (.1cm) node (u3) {} node[right] {$y_3$};
        \draw[fill=black] (0*\hor, 3*\v) circle (.1cm) node (v4) {} node[left] {$x_4$};
        \draw[fill=black] (1*\hor, 3*\v) circle (.1cm) node (u4) {} node[right] {$y_4$};

        \draw (v1) -- (u1);
        \draw (v1) -- (u2);
        \draw (v2) -- (u1);
        \draw (v2) -- (u2);
        \draw (v3) -- (u3);
        \draw (v3) -- (u4);
        \draw (v4) -- (u3);
        \draw (v4) -- (u4);

        \draw[dashed, line width=0.3mm, red] (v1) -- (u3);
        \draw[dashed, line width=0.3mm, red] (v1) -- (u4);
        \draw[dashed, line width=0.3mm, red] (v2) -- (u3);
        \draw[dashed, line width=0.3mm, red] (v2) -- (u4);
        \draw[dashed, line width=0.3mm, blue] (v3) -- (u2);
        \draw[dashed, line width=0.3mm, blue] (v3) -- (u1);
        \draw[dashed, line width=0.3mm, blue] (v4) -- (u1);
        \draw[dashed, line width=0.3mm, blue] (v4) -- (u2);
      \end{tikzpicture}
      \caption{}\label{2case5a}
    \end{subfigure}
    \quad
    \centering
    \begin{subfigure}[t]{2.5in}
      \centering
      \begin{tikzpicture}
        \def\hor{3cm}
        \def\v{-1cm}

        \draw[fill=black] (0*\hor, 0*\v) circle (.1cm) node (v1) {} node[left] {$x_1$};
        \draw[fill=black] (1*\hor, 0*\v) circle (.1cm) node (u1) {} node[right] {$y_1$};
        \draw[fill=black] (0*\hor, 1*\v) circle (.1cm) node (v2) {} node[left] {$x_2$};
        \draw[fill=black] (1*\hor, 1*\v) circle (.1cm) node (u2) {} node[right] {$y_2$};
        \draw[fill=black] (0*\hor, 2*\v) circle (.1cm) node (v3) {} node[left] {$x_3$};
        \draw[fill=black] (1*\hor, 2*\v) circle (.1cm) node (u3) {} node[right] {$y_3$};
        \draw[fill=black] (0*\hor, 3*\v) circle (.1cm) node (v4) {} node[left] {$x_4$};
        \draw[fill=black] (1*\hor, 3*\v) circle (.1cm) node (u4) {} node[right] {$y_4$};

	\draw (v1) -- (u1);
        \draw (v1) -- (u2);
        \draw (v2) -- (u1);
        \draw (v2) -- (u2);
        \draw (v3) -- (u3);
        \draw (v3) -- (u4);
        \draw (v4) -- (u3);
        \draw (v4) -- (u4);

        \draw (v1) -- (u3);
        \draw (v1) -- (u4);
        \draw (v2) -- (u3);
        \draw (v2) -- (u4);
        \draw (v3) -- (u2);
      \end{tikzpicture}
      \caption{}\label{2case5b}
    \end{subfigure}
    \caption{}
  \end{figure}

  \noindent where the two cycles are $x_1, y_1, x_2, y_2$ and $x_3, y_3, x_4, y_4$. The dashed lines form the two connected components of the bipartite complement. In order to minimally satisfy our assumptions, we must add one entire connected component from the bipartite complement and at least a single edge from the other yielding the graph in Figure~\ref{2case5b}. If less than a full component of the bipartite complement is included, there is an induced subgraph whose bipartite complement is not essentially connected, violating Theorem~\ref{GraphConditions}. If the extra edge is not present, then the bipartite complement contains a 4-cycle, also violating Theorem~\ref{GraphConditions}.

  Calling this graph in Figure~\ref{2case5b} $H$, one computes

  \begin{align*}
    I_H = (&e_{12}e_{21}-e_{11}e_{22},
             e_{13}e_{21}-e_{11}e_{23},
             e_{14}e_{21}-e_{11}e_{24},
             e_{13}e_{22}-e_{12}e_{23},
             e_{14}e_{22}-e_{12}e_{24},\\
             &e_{14}e_{23}-e_{13}e_{24},
             e_{13}e_{32}-e_{12}e_{33},
             e_{14}e_{32}-e_{12}e_{34},
             e_{23}e_{32}-e_{22}e_{33},
             e_{24}e_{32}-e_{22}e_{34},\\
             &e_{14}e_{33}-e_{13}e_{34},
             e_{24}e_{33}-e_{23}e_{34},
             e_{14}e_{43}-e_{13}e_{44},
             e_{24}e_{43}-e_{23}e_{44},
             e_{34}e_{43}-e_{33}e_{44}).
  \end{align*}

  \noindent Once again, the corresponding Koszul syzygy is a sum of linear syzygies:

  \scalebox{.85}{
    \vbox{
      \[ \mat{-e_{33}e_{44}+e_{34}e_{43} \\ 0 \\ 0 \\ 0 \\ 0 \\ 0 \\ 0 \\ 0 \\ 0 \\ 0 \\ 0 \\ 0 \\ 0 \\ 0 \\e_{11}e_{22} - e_{12} e_{21}}
        = e_{11}\mat{0 \\ 0 \\ 0 \\ 0 \\ 0 \\ 0 \\ 0 \\ 0 \\ 0 \\ 0 \\-e_{44} \\ e_{43} \\ 0 \\ -e_{32} \\ e_{22}}
        - e_{21}\mat{0\\0\\0\\0\\0\\-e_{44}\\e_{43}\\0\\0\\-e_{32}\\0\\0\\0\\0\\e_{12}}
        + e_{44}\mat{-e_{33} \\ e_{32} \\ 0\\0\\0\\-e_{21}\\0\\0\\0\\0\\e_{11}\\0\\0\\0\\0}
        - e_{43}\mat{-e_{34} \\ 0 \\ e_{32} \\ 0 \\ 0 \\0 \\ -e_{21} \\ 0 \\ 0 \\ 0 \\ 0 \\ e_{11} \\ 0 \\ 0 \\ 0}
        + e_{32}\mat{0 \\ -e_{44} \\ e_{43} \\ 0 \\ 0 \\ 0 \\ 0 \\ 0 \\ 0 \\ -e_{21} \\ 0 \\ 0 \\ 0 \\ e_{11} \\ 0}.
      \]
    }}

  %The reader can verify that the terms on the right-hand side are linear syzygies.

  Since all possible configurations of two distinct 4-cycles produce non-minimal Koszul syzygies, we can conclude that  $I_G$ satisfies $\mathbf{N}_2$.

 % Conversely, suppose $G$ is a bipartite graph of minimum degree at least 2 such that every cycle of length $\geq 6$ has a chord such that $\bar{G}$ is not a tree of diameter at most 3.  By Theorem~\ref{GraphConditions}, $G$ contains $H^{(i)}$ for some $1 \leq i \leq 8$ as an induced subgraph. By Lemmas~\ref{N2obstruction1} and \ref{N2obstruction2}, this induced subgraph $H^{(i)}$ has $\beta_{2,4}(I_{H^{(i)}}) \neq 0$, so by Lemma~\ref{InducedGraphBetti}, $\beta_{2,4}(I_G) \neq 0$, and $G$ does not satisfy condition $\mathbf{N}_2$. 
\end{proof}

\begin{rmk} The requirement that $\delta(G) \ge 2$ is  not restrictive.  If $G$ has an isolated (degree 0) vertex, then removing it does not change the edge ring $k[G]$.  If $G$ has a vertex of degree 1, then removing the adjacent edge merely reduces the embedding dimension of the ring of $I_G$; it does not change the minimal cycles in $G$ and thus does not affect the structure of $I_G$ or its resolution.  However, requiring $\delta(G) \ge 2$ makes our main result, which refers to $\bar{G}$, much easier to state and apply in practice.
\end{rmk}

\begin{rmk}A bipartite graph $G$ with $\delta(G) \geq 2$ such that every cycle of length $\geq 6$ has a chord can fail to satisfy Theorem~\ref{N2entire} in three different ways: $\overline{G}$ could be a tree of diameter greater than 3, $\overline{G}$ could be disconnected, or $\overline{G}$ could contain a cycle.   Figure~\ref{3graphs} gives an example of each type.  Dashed lines represent edges in $\bar{G}$.  The corresponding Betti tables are listed below showing the corresponding toric edge ideals are not linearly presented.

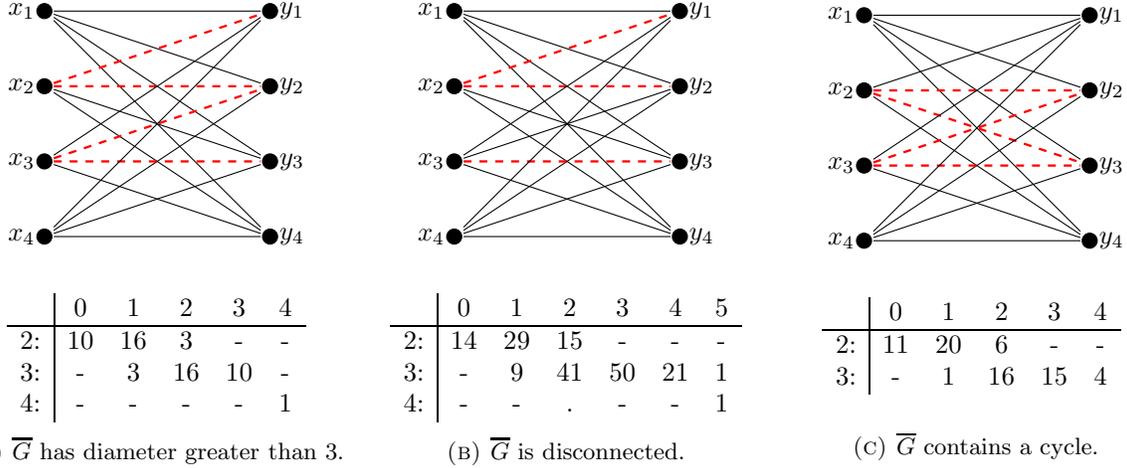
\begin{figure}[H]
\centering
\begin{subfigure}{.33\textwidth}
      \centering
      \begin{tikzpicture}
        \def\hor{3cm}
        \def\v{-1cm}

        \draw[fill=black] (0*\hor, 0*\v) circle (.1cm) node (v1) {} node[left] {$x_1$};
        \draw[fill=black] (1*\hor, 0*\v) circle (.1cm) node (u1) {} node[right] {$y_1$};
        \draw[fill=black] (0*\hor, 1*\v) circle (.1cm) node (v2) {} node[left] {$x_2$};
        \draw[fill=black] (1*\hor, 1*\v) circle (.1cm) node (u2) {} node[right] {$y_2$};
        \draw[fill=black] (1*\hor, 2*\v) circle (.1cm) node (u3) {} node[right] {$y_3$};
        \draw[fill=black] (0*\hor, 2*\v) circle (.1cm) node (v4) {} node[left] {$x_3$};
        \draw[fill=black] (1*\hor, 3*\v) circle (.1cm) node (u4) {} node[right] {$y_4$};
        \draw[fill=black] (0*\hor, 3*\v) circle (.1cm) node (v5) {} node[left] {$x_4$};

        \draw (v1) -- (u3);
        \draw (v1) -- (u4);
        \draw (v1) -- (u1);
        \draw (v1) -- (u2);
        \draw[red, dashed, line width=0.3mm] (v2) -- (u1);
        \draw[red, dashed, line width=0.3mm] (v2) -- (u2);
        \draw (v2) -- (u3);
        \draw (v2) -- (u4);
        \draw[red, dashed, line width=0.3mm] (v4) -- (u2);
        \draw (v5) -- (u2);
        \draw (v4) -- (u1);
        \draw[red, dashed, line width=0.3mm] (v4) -- (u3);
        \draw (v4) -- (u4);
        \draw (v5) -- (u1);
        \draw (v5) -- (u3);
        \draw (v5) -- (u4);
      \end{tikzpicture}
      
            \vspace{.5cm}

      \begin{tabular}{c|ccccc}
     &0&1&2&3&4\\
     \hline
     \text{2:}&10&16&3&\text{-}&\text{-}\\\text{3:}&\text{-}&3&16&10&\text{-}\\\text{4:}&\text{-}&\text{-}&\text{-}&\text{-}&
     1\\\end{tabular}
      \caption{$\overline{G}$ has diameter greater than 3.}
    \end{subfigure}% 
    \begin{subfigure}{.33\textwidth}
      \centering
      \begin{tikzpicture}
        \def\hor{3cm}
        \def\v{-1cm}

        \draw[fill=black] (0*\hor, 0*\v) circle (.1cm) node (v1) {} node[left] {$x_1$};
        \draw[fill=black] (1*\hor, 0*\v) circle (.1cm) node (u1) {} node[right] {$y_1$};
        \draw[fill=black] (0*\hor, 1*\v) circle (.1cm) node (v2) {} node[left] {$x_2$};
        \draw[fill=black] (1*\hor, 1*\v) circle (.1cm) node (u2) {} node[right] {$y_2$};
        \draw[fill=black] (1*\hor, 2*\v) circle (.1cm) node (u3) {} node[right] {$y_3$};
        \draw[fill=black] (0*\hor, 2*\v) circle (.1cm) node (v4) {} node[left] {$x_3$};
        \draw[fill=black] (1*\hor, 3*\v) circle (.1cm) node (u4) {} node[right] {$y_4$};
        \draw[fill=black] (0*\hor, 3*\v) circle (.1cm) node (v5) {} node[left] {$x_4$};

        \draw (v1) -- (u3);
        \draw (v1) -- (u4);
        \draw (v1) -- (u1);
        \draw (v1) -- (u2);
        \draw[red, dashed, line width=0.3mm] (v2) -- (u1);
        \draw[red, dashed, line width=0.3mm] (v2) -- (u2);
        \draw (v2) -- (u3);
        \draw (v2) -- (u4);
        \draw (v4) -- (u2);
        \draw (v5) -- (u2);
        \draw (v4) -- (u1);
        \draw[red, dashed, line width=0.3mm] (v4) -- (u3);
        \draw (v4) -- (u4);
        \draw (v5) -- (u1);
        \draw (v5) -- (u3);
        \draw (v5) -- (u4);
      \end{tikzpicture}
      
                  \vspace{.5cm}

      \begin{tabular}{c|cccccc}
      &0&1&2&3&4&5\\\hline\text{2:}&14&29&15&\text{-}&\text{-}&\text{-}\\\text{3:}&\text{-}&9&41&50&21&1\\\text{4:}&\text{-}&\text{-}&\text
      {.}&\text{-}&\text{-}&1\\\end{tabular}
      \caption{$\overline{G}$ is disconnected.}
    \end{subfigure}%
    \begin{subfigure}{.33\textwidth}
      \centering
      \begin{tikzpicture}
        \def\hor{3cm}
        \def\v{-1cm}

        \draw[fill=black] (0*\hor, 0*\v) circle (.1cm) node (v1) {} node[left] {$x_1$};
        \draw[fill=black] (1*\hor, 0*\v) circle (.1cm) node (u1) {} node[right] {$y_1$};
        \draw[fill=black] (0*\hor, 1*\v) circle (.1cm) node (v2) {} node[left] {$x_2$};
        \draw[fill=black] (1*\hor, 1*\v) circle (.1cm) node (u2) {} node[right] {$y_2$};
        \draw[fill=black] (1*\hor, 2*\v) circle (.1cm) node (u3) {} node[right] {$y_3$};
        \draw[fill=black] (0*\hor, 2*\v) circle (.1cm) node (v4) {} node[left] {$x_3$};
        \draw[fill=black] (1*\hor, 3*\v) circle (.1cm) node (u4) {} node[right] {$y_4$};
        \draw[fill=black] (0*\hor, 3*\v) circle (.1cm) node (v5) {} node[left] {$x_4$};

        \draw (v1) -- (u3);
        \draw (v1) -- (u4);
        \draw (v1) -- (u1);
        \draw (v1) -- (u2);
        \draw (v2) -- (u1);
        \draw[red, dashed, line width=0.3mm] (v2) -- (u2);
        \draw[red, dashed, line width=0.3mm] (v2) -- (u3);
        \draw (v2) -- (u4);
        \draw[red, dashed, line width=0.3mm] (v4) -- (u2);
        \draw (v5) -- (u2);
        \draw (v4) -- (u1);
        \draw[red, dashed, line width=0.3mm] (v4) -- (u3);
        \draw (v4) -- (u4);
        \draw (v5) -- (u1);
        \draw (v5) -- (u3);
        \draw (v5) -- (u4);
      \end{tikzpicture}
      
                  \vspace{.5cm}

      \begin{tabular}{c|ccccc}
      &0&1&2&3&4\\
      \hline
      \text{2:}&11&20&6&\text{-}&\text{-}\\\text{3:}&\text{-}&1&16&15&4\\
      \end{tabular}
\vspace{.31cm}
      
      \caption{$\overline{G}$ contains a cycle.}
    \end{subfigure}%
    \caption{Three examples of bipartite graphs with non-linearly presented toric edge ideals and their Betti tables}\label{3graphs}
    \end{figure}
    
    \end{rmk}

Consider the complete bipartite graph $K_{5,5}$.
%  \begin{figure}[H]
%    \centering
   %   \begin{tikzpicture}
    %    \def\hor{3cm}
    %    \def\v{-1cm}  
     %   \draw[fill=black] (0*\hor, 0*\v) circle (.1cm) node (v1) {} node[left] {$x_1$};
     %   \draw[fill=black] (1*\hor, 0*\v) circle (.1cm) node (u1) {} node[right] {$y_1$};
     %   \draw[fill=black] (0*\hor, 1*\v) circle (.1cm) node (v2) {} node[left] {$x_2$};
     %   \draw[fill=black] (1*\hor, 1*\v) circle (.1cm) node (u2) {} node[right] {$y_2$};
     %   \draw[fill=black] (0*\hor, 2*\v) circle (.1cm) node (v3) {} node[left] {$x_3$};
      %  \draw[fill=black] (1*\hor, 2*\v) circle (.1cm) node (u3) {} node[right] {$y_3$};
      %  \draw[fill=black] (0*\hor, 3*\v) circle (.1cm) node (v4) {} node[left] {$x_4$};
      %  \draw[fill=black] (1*\hor, 3*\v) circle (.1cm) node (u4) {} node[right] {$y_4$};
      %  \draw[fill=black] (0*\hor, 4*\v) circle (.1cm) node (v5) {} node[left] {$x_5$};
      %  \draw[fill=black] (1*\hor, 4*\v) circle (.1cm) node (u5) {} node[right] {$y_5$};
%        \draw (v1) -- (u1);
   %     \draw (v1) -- (u2);
      %  \draw (v2) -- (u1);
       % \draw (v2) -- (u2);
       % \draw (v2) -- (u1);
       % \draw (v2) -- (u3);
       % \draw (v3) -- (u1);
       % \draw (v1) -- (u3);
       % \draw (v3) -- (u2);
       % \draw (v3) -- (u3);
       % \draw (v1) -- (u4);
       % \draw (v1) -- (u5);
       % \draw (v2) -- (u4);
       % \draw (v2) -- (u5);
       % \draw (v3) -- (u4);
       % \draw (v3) -- (u5);
       % \draw (v4) -- (u1);
       % \draw (v4) -- (u2);
       % \draw (v4) -- (u3);
       % \draw (v4) -- (u4);
       % \draw (v4) -- (u5);
       % \draw (v5) -- (u1);
       % \draw (v5) -- (u2);
       % \draw (v5) -- (u3);
       % \draw (v5) -- (u4);
       % \draw (v5) -- (u5);
%      \end{tikzpicture}
  %  \end{figure}
    Then $\beta_{2, 5}(I_{K_{5,5}}) = 0$ if and only if the characteristic of $k$ is not 3.  
    In particular, in characteristic other than 3, $I_{K_{5,5}}$ has partial graded Betti table\\

    \begin{center}
      \begin{tabular}{c|cccc}
       &$0$&$1$&$2$ \\ \hline
        \text{$2$:}&$100$&$800$&$3075$& $\cdots$ \\
      \end{tabular},\\
    \end{center}

 \noindent   while in characteristic 3, $I_{K_{5,5}}$ has partial graded Betti table\\

    \begin{center}
      \begin{tabular}{c|cccc}
       &$0$&$1$&$2$ \\ \hline
        \text{$2$:}&$100$&$800$&$3075$& $\cdots$ \\
        \text{$3$:}&\text{-}&\text{-}&$1$ & $\cdots$ \\
      \end{tabular}.\\
    \end{center}
    
\noindent More generally, Hashimoto \cite{H} showed that the ideal $I_t(\mathbf{M})$ of $t \times t$ minors of a generic $m \times n$ matrix  has the same third Betti numbers independent of the characteristic if $t = 1$ or if $t \ge \min\{m,n\} - 2$, whereas the third Betti number is larger in characteristic $3$ if $2 \le t \le \min\{m,n\} - 3$.  The ideal $I_{K_{5,5}}$, corresponding to the $2 \times 2$ minors of a generic $5 \times 5$ matrix, is thus the minimal situation where $\beta_3(I_{K_{m,n}})$ depends on the characteristic.  While this example shows that characterizing toric edge ideals of bipartite graphs satisfying $\mathbf{N}_3$ must refer to the characteristic of the coefficient field, we show below that this is the only obstruction.

%complete with good characteristic (not 3) implies N3
\begin{thm}\label{N3} Let $G$ be a bipartite graph with $\delta(G) \ge 2$.  The ideal $I_G$ satisfies condition $\mathbf{N}_3$ if and only if $G = K_{m,n}$ for some $m,n$, unless the characteristic of $k$ is $3$ and $\min\{m, n\} \geq 5$. 
\end{thm}

%\begin{proof}
%TODO
%\end{proof}

%some complete with bad characteristic (yes 3) implies N3
%\begin{prop}
%If the characteristic of $k$ is 3 and $G$ is a $(n, m)$-complete bipartite graph with $\delta(G) \geq 2$, then $I_G$ satisfies $\mathbf{N}_3$ if and only if $\min\{m, n\} \leq 4$.
%\end{prop}

% N3 implies complete
%\begin{prop} Let $G = (V,E)$ be a connected, bipartite graph with $\delta(G) \ge 2$.  If $G$ is not a complete bipartite graph, then $I_G$ does not satisfy property $\mathbf{N}_3$.
%\end{prop}

\begin{proof}
First suppose that $G = K_{5,5}$.  In characteristic $0$, $I_G$ satisfies $\mathbf{N}_3$ by considering the resolution of Lascoux or Pragacz and Weyman.  It follows by considering the number of boxes in the last partition of the Lascoux complex (20 in this case), the resolution is the same in characteristic $p > 20$ as it is in characteristic $0$.  A quick Macaulay2 \cite{M2} calculation shows that $p = 3$ is the only characteristic less than $20$ which in which $I_G$ fails to satisfy $\mathbf{N}_3$.  

Now let $G = K_{m,n}$ for arbitrary $m,n$.  If $\min\{m,n\} \le 4$, then $I_G$ has the same graded Betti numbers as its characteristic 0 Lascoux resolution by \cite{HK} and \cite{ABW}.  Thus we may assume $\min\{m,n\} \ge 5$.  If the $\mathrm{char}(k) = 3$, then $G$ has $K_{5,5}$ as an induced subgraph and so does not satisfy $\mathbf{N}_3$ as above.  Thus we may assume $\mathrm{char}(k) \neq 3$.  If $I_G$ does not satisfy $\mathbf{N}_3$, then $\beta_{2,j}(I_G) \neq 0$ for some $j \ge 5$.  Since $S/I_G$ is Koszul and since $I_G$ is linearly presented, it follows from \cite[Main Theorem (2)]{ACI} that $\beta_{2,j}(I_G) = 0$ for all $j \ge 6$.  Thus if $I_G$ fails to satisfy $\mathbf{N}_3$, we must have $\beta_{2,5}(I_G) \neq 0$.

Now by Theorem~\ref{InducedGraphBetti} there must be an induced subgraph $H$ of $G$ with at most $10$ vertices that fails $\mathbf{N}_3$.  Since $H$ is an induced subgraph, it is also a complete bipartite graph, say $K_{m',n'}$ with $m'+n' = 10$.  If $\min\{m',n'\} \le 4$, then $I_H$ has the same graded Betti numbers as that characteristic $0$ Lascoux resolution, which satisfies $\mathbf{N}_3$ by \cite{HK} and \cite{ABW}.  Thus is suffices to consider the case $H = K_{5,5}$, as we have above.

Now suppose $G$ is not a complete bipartite graph.
  If $I_G$ does not satisfy $\mathbf{N}_2$, then it doesn't satisfy $\mathbf{N}_3$, so it is enough to consider a graph $G$ such that $I_G$ satisfies $\mathbf{N}_2$. By Theorem~\ref{N2entire}, $\overline{G}$ is essentially a tree of diameter at most 3. In particular, since $G$ is not a complete graph, $\overline{G}$ has at least one edge. Moreover, since $\bar{G}$ is a tree of diameter at most $3$, there is an edge with vertices $x$ and $y$ such that every other edge of $\bar{G}$ is adjacent to this edge.  Since $\delta(G) \geq 2$, there exist vertices $x', x'', y', y''$ such that $\{x, y'\}, \{x, y''\}, \{y, x'\},$ and $\{y, x''\}$ are all edges in $G$. Since every edge in $\bar{G}$ is adjacent to $\{x,y\}$, we must also have the edges $\{x', y'\}, \{x', y''\}, \{x'', y'\},$ and $\{x'', y''\}$ in $G$. So the induced subgraph on $\{x, x', x'', y, y', y''\}$ is the graph $H$ in Lemma~\ref{N3obstruction}, which satisfies $\beta_{2, 5}(I_H) \neq 0$. By Theorem~\ref{InducedGraphBetti}, $\beta_{2, 5}(I_G) \neq 0$, so $I_G$ does not satisfy $\mathbf{N}_3$.
\end{proof}

The preceding work yields the following surprising characterization of toric edge ideals with linear free resolutions.  That $I_{K_{2,n}}$ has a linear free resolution is well-known and follows from the Eagon-Northcott resolution.  Ohsugi and Hibi \cite[Theorem 4.6]{OH1} showed that $I_G$ has a linear free resolution if and only if $G = K_{2,n}$ for some $n$.  The content of the following theorem is that the only bipartite graphs satisfying condition $\mathbf{N}_4$ are complete bipartite graphs $K_{2,n}$ for some $n$ and that condition $\mathbf{N}_4$ is sufficient to guarantee linear free resolutions regardless of the characteristic.

% N4 implies (2,n)
\begin{thm}\label{N4} Let $G $ be a bipartite graph with $\delta(G) \ge 2$.  The ideal $I_G$ satisfies property $\mathbf{N}_4$ if and only if $G = K_{2,n}$ for some $n$.  In this case, $I_G$ has a linear free resolutions and thus satisfies $\mathbf{N}_p$ for all $p \ge 1$.  %If $G$ is not a $(2,n)$-complete bipartite graph, then $I_G$ does not satisfy property $\mathbf{N}_4$.
\end{thm}

\begin{proof}
  If $I_G$ does not satisfy $\mathbf{N}_3$, then it doesn't satisfy $\mathbf{N}_4$, so it is enough to consider a graph $G$ such that $I_G$ satisfies $\mathbf{N}_3$. So we can assume $G$ is a complete bipartite graph. If $G = K_{m, n}$ with $m, n \geq 3$, then it has $C = K_{3, 3}$ as an induced subgraph. By Lemma~\ref{N4obstruction}, $\beta_{3,6}(I_C) \neq 0$, so by Theorem~\ref{InducedGraphBetti} $\beta_{3,6}(I_G) \neq 0$.
\end{proof}

Keeping in mind Theorem~\ref{KoszulSchreyer}, we see that these proofs relied on the fact that $I_G$ was generated by a Gr\"obner basis of quadrics. If we attempt to directly extend these results to toric edge ideals of arbitrary (not necessarily bipartite) graphs, we lose the power of this result, even in the case where $I_G$ is generated by quadrics. For example the graph pictured below

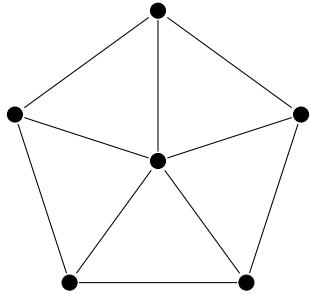
\begin{figure}[H]
      \centering
      \begin{tikzpicture}
        \def\hor{2cm}
        \def\v{2cm}
        \def\a{90}

        \draw[fill=black] ({\hor*cos(0+\a)}, {\v*sin(0+\a)}) circle (.1cm) node (v1) {} node[above] {};
        \draw[fill=black] ({\hor * cos(72+\a) } , {\v * sin(72+\a)}) circle (.1cm) node (v5) {} node[left] {};
        \draw[fill=black] ({\hor * cos(288+\a)}, {\v * sin(288+\a)}) circle (.1cm) node (v2) {} node[right] {};
        \draw[fill=black] ({\hor * cos(144+\a)}, {\v * sin(144+\a)}) circle (.1cm) node (v4) {} node[below left] {};
        \draw[fill=black] ({\hor * cos(216+\a)}, {\v * sin(216+\a)}) circle (.1cm) node (v3) {} node[below right] {};
        \draw[fill=black] (0,0) circle (.1cm) node (v6) {} node[below = .1cm] {};

        \draw (v1) -- (v2);
        \draw (v2) -- (v3);
        \draw (v3) -- (v4);
        \draw (v4) -- (v5);
        \draw (v5) -- (v1);
        \draw (v1) -- (v6);
        \draw (v2) -- (v6);
        \draw (v3) -- (v6);
        \draw (v4) -- (v6);
        \draw (v5) -- (v6);
      \end{tikzpicture}
     % \caption{Two 4-cycles which share an edge.}\label{3case1}
    \end{figure}
  \noindent  has toric edge ideal generated by quadrics \cite[Example 5.28]{HHO} but  has no Gr\"obner basis of quadrics with respect to any monomial ordering \cite[Example 1.18]{HHO}.

%%%%%%%%%%%%%%%%%%%%%%%%%%%%%%%%%%%%%%%%%%%%%%%%%%%%%%%%%%%%%%%%%%%%%%%%%%%%%%%%%%%%%%%%%%%%%%
\section{Linear Syzygies of Polyominoes}\label{poly}

Polyominoes and their associated ideals were introduced by Qureshi, where it was shown that the ring associated to a convex polyomino is normal and Cohen-Macaulay \cite[Theorem 2.2]{Q}.  Later work by Ene, Herzog, and Hibi showed that the defining ideals are generated by a quadratic Gr\"obner basis by viewing them as toric ideals associated to bipartite graphs \cite[Proposition 2.3]{EHH}.  They also give a characterization of polynominoes whose associated ideals are linearly presented \cite[Theorem 3.1]{EHH}.  However, when we translate our result on bipartite graphs whose toric edge ideals are linearly presented, we discovered a discrepancy; in particular, there are polyominoes that are not linearly presented that satisfy \cite[Theorem 3.1]{EHH}.  The purpose of this section is to then translate our results on toric edge ideals of bipartite graphs into results on convex polynomino ideals satisfying coditions $\mathbf{N}_p$ for all $p$, thereby correcting the error in the above theorem.  We begin with some notation.

If $a,b \in \mathbb{N}^2$ with $a \le b$ under the natural partial order, the set $[a,b] = \{c \in \N^2 \,|\, a \le c \le b\}$ is called an interval.  If $b = a + (1,1)$, then $[a,b]$ is called a cell.    The edges of the cell $C = [a, a+(1,1)]$ are the sets $\{a, a+(0,1)\}, \{a + (0,1), a+(1,1)\}, \{a + (1,1), a+(1,0)\}, \{a + (1,0), a\}$ and the points $a, a+(0,1), a+(1,1), a+(1,0)$ are the vertices of $C$.   The vertex $a$ is called the lower left corner of $C$.   Let $\mathcal{P}$ be a finite collections of cells.  The set of vertices $V(\mathcal{P})$ is the union of the sets of vertices of all cells in $\mathcal{P}$.   If $C,D \in \mathcal{P}$, then $C$ and $D$ are connected  if there is a sequence of cells of $\mathcal{P}$  given by $C = C_1,\ldots,C_t =D$ such that $C_i \cap C_{i+1}$ is an edge of $C_i$ for $i =1\ldots,t-1$.  A collection of cells $\mathcal{P}$ is a polyomino if any two of its cells are connected.  Two polyominos are isomorphic is they are mapped to each other by a finite sequence of translations, rotations, and reflections.  A polyomino $\mathcal{P}$ is row convex if given any any two cells of $\mathcal{P}$ with lower left corners $(i_1,j)$ and $(i_2,j)$ with $i_1 < i_2$, all of the cells with lower left corners $(i,j)$ with $i_1 < i < i_2$ are also in $\mathcal{P}$.  Similarly, one defines $\mathcal{P}$ to be column convex if given any two cells of $\mathcal{P}$ with lower left corners $(i,j_1)$ and $(i,j_2)$ with $j_1 < j_2$, one has that all the cells with lower left corners $(i,j)$ with $j_1 < j < j_2$ are in $\mathcal{P}$.  Finally $\mathcal{P}$ is convex if it is row convex and column convex.

Now let $\mathcal{P}$ be a polyomino.  We may rotate and translate $\mathcal{P}$ until $[(1,1),(m,n)]$ is the smallest interval containing $\mathcal{P}$.  Fix a field $k$ and a polynomial ring $S = k[x_{i,j}\,|\,1 \le i \le m, 1 \le j \le n]$.  The polyomino ideal $I_\mathcal{P}$ is the ideal of $S$ generated by the $2 \times 2$ minors $x_{ij}x_{kl} - x_{il}x_{kj}$ for with $[(i,j),(k,l)] \subset V(\mathcal{P})$.  

\begin{rmk}
In \cite[Proposition 2.3]{EHH}, it is shown every convex polyomino ideal is the toric edge ideal of a bipartite graph with a quadratic Gr\"obner basis and thus satisfies condition $\mathbf{N}_1$.  One identifies the vertical line segments with one set of vertices $x_i$ and the horizontal line segments with another set of vertices $y_j$; then one draws an edge between $x_i$ and $y_j$ if the corresponding line segments intersect.  Thus every convex polynomino corresponds to a bipartite graph such that every cycle of length $\ge 6$ has a chord.  

Here we remark that the converse does not hold; that is, not every quadratically generated toric edge ideal is the polyomino ideal for some convex polyomino.  The ideal $I_{H^{(1)}}$ which corresponds to a complete intersection of two quadratic binomials is clearly not associated to any convex polyomino, which cannot have exactly 2 minimal generators.  However, the disconnected collection of cells in Figure~\ref{cells} has an isomorphic  ideal of inner minors to $I_{H^{(1)}}$.
\begin{figure}[H]
\centering

\begin{tikzpicture}
\def\size{1cm}
\def\hor{\size}
\def\v{-1*\size}

\draw [line width=0.35mm] (0*\hor, 0) -- (1*\hor, 0);
\draw [line width=0.35mm] (1*\hor, 2*\v) -- (1*\hor, 0);
\draw [line width=0.35mm] (1*\hor, 2*\v) -- (2*\hor, 2*\v);
\draw [line width=0.35mm] (2*\hor, 1*\v) -- (2*\hor, 2*\v);
\draw [line width=0.35mm] (2*\hor, 1*\v) -- (0*\hor, 1*\v);
\draw [line width=0.35mm] (0*\hor, 0*\v) -- (0*\hor, 1*\v);

%\draw [line width=0.35mm] (1*\hor + 3*\hor, 0) -- (2*\hor+ 3*\hor, 0);
%\draw [line width=0.35mm] (0+ 3*\hor, 1*\v) -- (0+ 3*\hor, 2*\v); 
%\draw [line width=0.35mm] (0+ 3*\hor, 2*\v) -- (2*\hor+ 3*\hor, 2*\v);
%\draw [line width=0.35mm] (2*\hor+ 3*\hor, 2*\v) -- (2*\hor+ 3*\hor, 0);
%\draw [line width=0.35mm] (0+ 3*\hor, 1*\v) -- (.2*\hor+ 3*\hor, 1*\v);
%\draw [line width=0.35mm] (.2*\hor+ 3*\hor, 1*\v) -- (.2*\hor+ 3*\hor, .2*\v);
%\draw [line width=0.35mm] (1*\hor+ 3*\hor, .2*\v) -- (.2*\hor+ 3*\hor, .2*\v);
%\draw[line width=0.35mm] (1*\hor+ 3*\hor,.2*\v) -- (1*\hor+ 3*\hor,0);
 %       \draw[fill=black] (.2*\hor+ 3*\hor,.2*\v) circle (.1cm) node (v6) {} node[below = .1cm] {};

%\draw [line width=0.35mm] (.2*\hor+ 3*\hor, 1*\v) -- (2*\hor+ 3*\hor, 1*\v);
%\draw [line width=0.35mm] (.2*\hor+ 3*\hor, .8*\v) -- (2*\hor+ 3*\hor, .8*\v);
%\draw [line width=0.35mm] (.2*\hor+ 3*\hor, .6*\v) -- (2*\hor+ 3*\hor, .6*\v);
%\draw [line width=0.35mm] (.2*\hor+ 3*\hor, .4*\v) -- (2*\hor+ 3*\hor, .4*\v);
%\draw [line width=0.35mm] (.2*\hor+ 3*\hor, .2*\v) -- (2*\hor+ 3*\hor, .2*\v);

\end{tikzpicture}
\caption{}\label{cells}
\end{figure}

\end{rmk}

\begin{prop} Let $G = (V,E)$ is a connected bipartite graph with $\delta(G) \ge 2$ such that every chord with length $\ge 6$ has a chord and such that $\overline{G}$ is essentially a tree of diameter at most $3$.  Then there is a convex polynomo $\mathcal{P}$ such that $I_G$ and $I_\mathcal{P}$ are isomorphic.
\end{prop}

\begin{proof} Let $\{x_1,\ldots,x_m\} \sqcup \{y_1,\ldots,y_n\}$ be the vertex set of $G$.  Since $\bar{G}$ is essentially a tree of diameter at most $3$, if $G$ is not a complete bipartite graph, then $\bar{G}$ has an edge adjacent to every other edge.  By relabeling, we may assume this edge is $\{x_m,y_n\}$ and every other edge is of the form $\{x_m,y_j\}$ for some $1 \le j < n$ or $\{x_i,y_n\}$ for some $1 \le i < m$.  Again by relabeling vertices, we may assume that $\{x_i,y_n\} \in E$ if and only if $i \le m$ and $\{x_m,y_j\} \in E$ if and only if $j \le n'$ for some integers $m'$ and $n'$ with $1\le m' < m$ and $1 \le n' < n$.  It follows that $I_G = I_\mathcal{P}$ where $\mathcal{P}$ is the following polyomino:

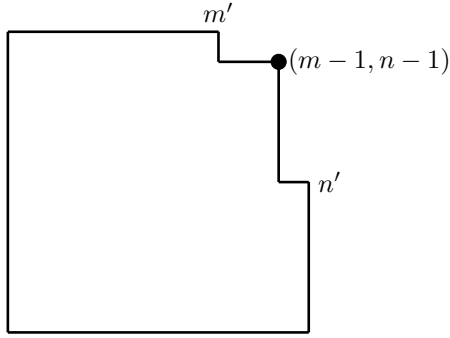
\begin{figure}[H]
\centering

\begin{tikzpicture}
\def\size{2cm}
\def\hor{\size}
\def\v{-1*\size}

\draw [line width=0.35mm] (0*\hor, 0) -- (1.4*\hor, 0);
\draw [line width=0.35mm] (1.4*\hor, .2*\v) -- (1.4*\hor, 0);
\draw [line width=0.35mm] (1.4*\hor, .2*\v) -- (1.8*\hor, .2*\v);

%\draw [line width=0.35mm] (1.6*\hor, .2*\v) -- (1.6*\hor, 0);
%\draw [line width=0.35mm] (1.6*\hor, .2*\v) -- (1.8*\hor,.2*\v);

\draw [line width=0.35mm] (0, 0*\v) -- (0, 2*\v); 
%\draw [line width=0.35mm] (.2*\hor, 1.5*\v) -- (0, 1.5*\v); 
%\draw [line width=0.35mm] (.2*\hor, 1.5*\v) -- (.2*\hor, 2*\v); 

\draw [line width=0.35mm] (0*\hor, 2*\v) -- (2*\hor, 2*\v);
\draw [line width=0.35mm] (2*\hor, 2*\v) -- (2*\hor, 1*\v);
\draw [line width=0.35mm] (1.8*\hor, .2*\v) -- (1.8*\hor, 1*\v);
\draw [line width=0.35mm] (2*\hor, 1*\v) -- (1.8*\hor, 1*\v);
%\draw [line width=0.35mm] (1.8*\hor, 1.6*\v) -- (2*\hor, 1.6*\v);
%\draw [line width=0.35mm] (1.8*\hor, 1.6*\v) -- (1.8*\hor, 2*\v);

%\draw [line width=0.35mm] (0, .5*\v) -- (.2*\hor, .5*\v);
%\draw [line width=0.35mm] (.2*\hor, .5*\v) -- (.2*\hor, .2*\v);
%\draw [line width=0.35mm] (.6*\hor, .2*\v) -- (0*\hor, .2*\v);
%\draw[line width=0.35mm] (.6*\hor,.2*\v) -- (.6*\hor,0);
        \draw[fill=black] (1.8*\hor,.2*\v) circle (.1cm) node (v6) {} node[below = .1cm] {} node[right] {$(m-1,n-1)$};
 \draw[fill=black] (1.4*\hor,0*\v) circle (.001cm) node (v6) {} node[below = .1cm] {} node[above] {$m'$};
 \draw[fill=black] (2*\hor,1*\v) circle (.001cm) node (v6) {} node[below = .1cm] {} node[right] {$n'$};
%\draw [line width=0.35mm] (1*\hor + 3*\hor, 0) -- (2*\hor+ 3*\hor, 0);
%\draw [line width=0.35mm] (0+ 3*\hor, 1*\v) -- (0+ 3*\hor, 2*\v); 
%\draw [line width=0.35mm] (0+ 3*\hor, 2*\v) -- (2*\hor+ 3*\hor, 2*\v);
%\draw [line width=0.35mm] (2*\hor+ 3*\hor, 2*\v) -- (2*\hor+ 3*\hor, 0);
%\draw [line width=0.35mm] (0+ 3*\hor, 1*\v) -- (.2*\hor+ 3*\hor, 1*\v);
%\draw [line width=0.35mm] (.2*\hor+ 3*\hor, 1*\v) -- (.2*\hor+ 3*\hor, .2*\v);
%\draw [line width=0.35mm] (1*\hor+ 3*\hor, .2*\v) -- (.2*\hor+ 3*\hor, .2*\v);
%\draw[line width=0.35mm] (1*\hor+ 3*\hor,.2*\v) -- (1*\hor+ 3*\hor,0);
 %       \draw[fill=black] (.2*\hor+ 3*\hor,.2*\v) circle (.1cm) node (v6) {} node[below = .1cm] {};

%\draw [line width=0.35mm] (.2*\hor+ 3*\hor, 1*\v) -- (2*\hor+ 3*\hor, 1*\v);
%\draw [line width=0.35mm] (.2*\hor+ 3*\hor, .8*\v) -- (2*\hor+ 3*\hor, .8*\v);
%\draw [line width=0.35mm] (.2*\hor+ 3*\hor, .6*\v) -- (2*\hor+ 3*\hor, .6*\v);
%\draw [line width=0.35mm] (.2*\hor+ 3*\hor, .4*\v) -- (2*\hor+ 3*\hor, .4*\v);
%\draw [line width=0.35mm] (.2*\hor+ 3*\hor, .2*\v) -- (2*\hor+ 3*\hor, .2*\v);

\end{tikzpicture}
\caption{A linearly related polyomino.}\label{linearpoly}
\end{figure}

\end{proof}

Using the previous dictionary between polyomino ideals and toric edge ideals of bipartite graphs, we translate our main Theorem~\ref{main2} to characterize polyomino ideals satisfying the various Green-Lazarsfeld conditions.

\begin{thm} Let $\mathcal{P}$ be a convex polyomino and let $k$ be a field.  
\begin{enumerate}
\item $I_\mathcal{P}$ satisfies property $\mathbf{N}_1$.
\item $I_\mathcal{P}$ satisfies property $\mathbf{N}_2$ if and only if $\mathcal{P}$ is isomorphic to a polyomino all of whose missing cells are in the first row or first column (possibly after rotating $\mathcal{P}$. See Figure~\ref{linearpoly2}.)
\item $I_\mathcal{P}$ satisfies property $\mathbf{N}_{3}$ if and only if $\mathcal{P}$ is an interval unless $\mathrm{char}(k) = 3$ and $\mathcal{P}$ is an interval with width and length at least 4.
\item $I_\mathcal{P}$ satisfies property $\mathbf{N}_{p}$ for some/any $p \ge 4$ if and only if $\mathcal{P}$ is an interval of the form $[a, (2,n) + a]$ for some $a \in \N^2$.
\end{enumerate}

\end{thm}

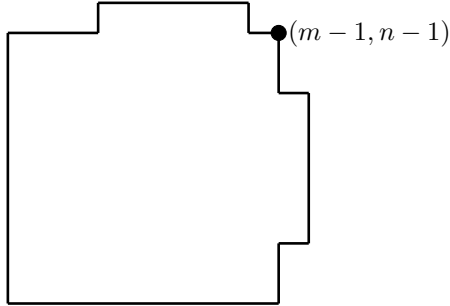
\begin{figure}[H]
\centering

\begin{tikzpicture}
\def\size{2cm}
\def\hor{\size}
\def\v{-1*\size}

\draw [line width=0.35mm] (.6*\hor, 0) -- (1.6*\hor, 0);
\draw [line width=0.35mm] (1.6*\hor, .2*\v) -- (1.6*\hor, 0);
\draw [line width=0.35mm] (1.6*\hor, .2*\v) -- (1.8*\hor,.2*\v);

\draw [line width=0.35mm] (0, .2*\v) -- (0, 2*\v); 
%\draw [line width=0.35mm] (.2*\hor, 1.5*\v) -- (0, 1.5*\v); 
%\draw [line width=0.35mm] (.2*\hor, 1.5*\v) -- (.2*\hor, 2*\v); 

\draw [line width=0.35mm] (0*\hor, 2*\v) -- (1.8*\hor, 2*\v);
\draw [line width=0.35mm] (2*\hor, 1.6*\v) -- (2*\hor, .6*\v);
\draw [line width=0.35mm] (1.8*\hor, .2*\v) -- (1.8*\hor, .6*\v);
\draw [line width=0.35mm] (2*\hor, .6*\v) -- (1.8*\hor, .6*\v);
\draw [line width=0.35mm] (1.8*\hor, 1.6*\v) -- (2*\hor, 1.6*\v);
\draw [line width=0.35mm] (1.8*\hor, 1.6*\v) -- (1.8*\hor, 2*\v);

%\draw [line width=0.35mm] (0, .5*\v) -- (.2*\hor, .5*\v);
%\draw [line width=0.35mm] (.2*\hor, .5*\v) -- (.2*\hor, .2*\v);
\draw [line width=0.35mm] (.6*\hor, .2*\v) -- (0*\hor, .2*\v);
\draw[line width=0.35mm] (.6*\hor,.2*\v) -- (.6*\hor,0);
        \draw[fill=black] (1.8*\hor,.2*\v) circle (.1cm) node (v6) {} node[below = .1cm] {} node[below = .1cm] {} node[right] {$(m-1,n-1)$};

%\draw [line width=0.35mm] (1*\hor + 3*\hor, 0) -- (2*\hor+ 3*\hor, 0);
%\draw [line width=0.35mm] (0+ 3*\hor, 1*\v) -- (0+ 3*\hor, 2*\v); 
%\draw [line width=0.35mm] (0+ 3*\hor, 2*\v) -- (2*\hor+ 3*\hor, 2*\v);
%\draw [line width=0.35mm] (2*\hor+ 3*\hor, 2*\v) -- (2*\hor+ 3*\hor, 0);
%\draw [line width=0.35mm] (0+ 3*\hor, 1*\v) -- (.2*\hor+ 3*\hor, 1*\v);
%\draw [line width=0.35mm] (.2*\hor+ 3*\hor, 1*\v) -- (.2*\hor+ 3*\hor, .2*\v);
%\draw [line width=0.35mm] (1*\hor+ 3*\hor, .2*\v) -- (.2*\hor+ 3*\hor, .2*\v);
%\draw[line width=0.35mm] (1*\hor+ 3*\hor,.2*\v) -- (1*\hor+ 3*\hor,0);
 %       \draw[fill=black] (.2*\hor+ 3*\hor,.2*\v) circle (.1cm) node (v6) {} node[below = .1cm] {};

%\draw [line width=0.35mm] (.2*\hor+ 3*\hor, 1*\v) -- (2*\hor+ 3*\hor, 1*\v);
%\draw [line width=0.35mm] (.2*\hor+ 3*\hor, .8*\v) -- (2*\hor+ 3*\hor, .8*\v);
%\draw [line width=0.35mm] (.2*\hor+ 3*\hor, .6*\v) -- (2*\hor+ 3*\hor, .6*\v);
%\draw [line width=0.35mm] (.2*\hor+ 3*\hor, .4*\v) -- (2*\hor+ 3*\hor, .4*\v);
%\draw [line width=0.35mm] (.2*\hor+ 3*\hor, .2*\v) -- (2*\hor+ 3*\hor, .2*\v);

\end{tikzpicture}
\caption{A general, linearly related, convex polyomino.}\label{linearpoly2}
\end{figure}

%%%%%%%%%%%%%%%%%%%%%%%%%%%%%%%%%%%%%%%%%%%%%%%%%%%%%%%%%%%%%%%%%%%%%%%%%%%%%%%%%%%%%%%%%%%%%%
\section{Application to a Question of Constantinescu, Kahle, and Varbaro}\label{reg}

Our original motivation for studying linearly presented toric edge ideals comes from the following question of Constantinescu, Kahle, and Varbaro \cite{CKV}:

\begin{qst}[{\cite[Question 1.1]{CKV}}]\label{Qckv}
Is there a family of linearly presented, quadratically generated ideals $\{I_n\subseteq R = k[x_1, . . . , x_n]\}_{n \in \mathbb{N}}$ such that 
\[\lim_{n \to \infty} \frac{\reg(I_n)}{n} > 0?\]
\end{qst}

\noindent A similar question in the Koszul setting was posed by Conca \cite[Question 2.8]{C}.  Such a family of ideals would have regularity growing linearly with respect to the number of variables.  In \cite{CKV}, Constantinescu, Kahle, and Varbaro construct a family of squarefree, quadratic monomial ideals with linear syzygies for arbitrarily many steps (i.e. satisfying property $\mathbf{N}_{p}$ for arbitrary $p$ if we ignore the normality condition) and with arbitrarily large regularity, however these ideals have a very large number of variables.  A result of Dao, Huneke, and Schweig \cite{DHS} shows that the regularity of squarefree monomial ideals with linear syzygies is bounded logarithmically in terms of the number of variables; in particular, no such families of ideals yielding a positive answer to Question~\ref{Qckv} can be monomial. 

Note that if $\depth(R/I_n) > 0$, we can mod out by a general linear form, thereby reducing the number of variables while preserving the graded Betti numbers and regularity.  It follows from the Auslander-Buchsbaum formula that the following question is equivalent to Question~\ref{Qckv}

\begin{qst}\label{Q2ckv}
Is there a family of linearly presented, quadratically generated ideals $\{I_n\subseteq R = k[x_1, . . . , x_n]\}_{n \in \mathbb{N}}$ such that 
\[\lim_{n \to \infty} \pd(I_n) = \infty \quad \text{and} \quad \lim_{n \to \infty} \frac{\reg(I_n)}{\pd(I_n)} > 0?\]
\end{qst}

\noindent The restriction that the ideals is linearly presented rules out complete intersections of $n$ quadrics for with $\pd(S/I_n) = \reg(S/I_n) = n$.    In general, both questions are still open.  A corollary to our Theorem~\ref{N2entire}  is that no such families exist among toric edge ideals associated to bipartite graphs.

\begin{cor}
There are no families of graphs $G_n$, where $G_n$ is bipartite and $I_{G_n}$ satisfies property $\mathbf{N}_{2}$, that give a positive answer to Question~\ref{Q2ckv}.  In other words, if $\lim_{n \to \infty} \pd(I_{G_n}) = \infty$, then  $\lim_{n \to \infty} \frac{\reg(I_{G_n})}{\pd(I_{G_n})} = 0$.
\end{cor}

\begin{proof} Fix a bipartite graph $G = (X \sqcup Y, E)$ such that $I_G$ is linearly presented and $\delta(G) \ge 2$.  Set $r = |X|$ and $s = |Y|$ and without loss of generality assume $2 \le r \le s$.  Since $I_G$ is Cohen-Macaulay, $\pd(S/I_G) = \mathrm{ht}(I_G).$  When $G$ is a complete $(r,s)$-bipartite graph, it is well-known that $\mathrm{ht}(I_G) = (r-1)(s-1)$.   If $G$ is an arbitrary bipartite graph such that $I_G$ is linearly presented, it follows from Theorem~\ref{N2entire} that $\bar{G}$ is a tree of diameter at most $3$.  Thus there are at most $(r-3) + (s-3) + 1$ edges missing from the complete bipartite graph and so $|E| \ge rs -r -s + 5$.  It follows that $\mathrm{ht}(I_G) \ge (r-1)(s-1) - r - s + 5 = (r-2)(s-1) - r + 4$.  By \cite[Theorem 4.9]{BOV}, $\reg(S/I_G) \le r$.  If the above limit is nonzero, we must have $\lim_{n \to \infty} r = \infty$.  Since $r \le s$, we get
\[\frac{\reg(I_{G_n})}{\pd(I_{G_n})} \le  \frac{r}{(r-2)(s-1) - r + 4} =  \frac{1}{\frac{r-2}{r}(s-1) - 1 + \frac{4}{r}} \to 0\]
as $n \to \infty$.
\end{proof}

It is worth noting that even though quadratic toric edge ideals of bipartite graphs are generated by quadratic Gr\"obner bases, this does not reduce the problem of answering the above question to the monomial case.  Indeed there are linearly presented, quadratic toric edge ideals whose lead term (monomial) ideals are not linearly presented.  For a simple example, let $G$ be $K_{4,3}$ with one edge removed.  By Theorem~\ref{N2entire}, $I_G$ is linearly presented.  One checks however that $LT(I_G)$ is quadratic but not linearly presented.\\

\noindent \textbf{Acknowledgements}\\
The second author was partially supported by a grant from the Simons Foundation (576107, JGM) and NSF grant DMS--1900792.


\begin{thebibliography}{99}

\bibitem{ABW}{K. Akin, D. Buchsbaum, J. Weyman, 
Resolutions of determinantal ideals: the submaximal minors. 
Adv. in Math. {\bf 39} (1981), no. 1, 1--30. }

\bibitem{ACI}{L. Avramov, A. Conca and S. Iyengar, Free resolutions over commutative Koszul algebras. 
Math. Res. Lett. {\bf 17} (2010), no. 2, 197--210. }

\bibitem{BOV}{J. Biermann, A. O'Keefe, A. Van Tuyl,
Bounds on the regularity of toric ideals of graphs. 
Adv. in Appl. Math. {\bf 85} (2017), 84--102.}

\bibitem{BE}{D.A. Buchsbaum, D. Eisenbud, Algebra structures for finite free resolutions, and some structure theorems for ideals of codimension 3, Am. J. Math. {\bf 99} (3) (1977) 447--485.}

\bibitem{C}{A. Conca, 
Koszul algebras and their syzygies. Combinatorial algebraic geometry, 1--31, 
Lecture Notes in Math., 2108, Fond. CIME/CIME Found. Subser., Springer, Cham, 2014. }

\bibitem{CKV}{A. Constantinescu, T. Kahle, M. Varbaro, Linear syzygies, hyperbolic Coxeter groups and regularity. 
Compos. Math. {\bf 155} (2019), no. 6, 1076--1097. }

\bibitem{CLO}{D. Cox, J. Little, D. O'Shea, 
Using algebraic geometry, 
Second edition. Graduate Texts in Mathematics, 185. Springer, New York, 2005.}


\bibitem{DHS}{H. Dao, C. Huneke, J. Schweig, Bounds on the regularity and projective dimension of ideals associated to graphs, Journal of Algebraic Combinatorics {\bf 38} (2013), no. 1, 37--55.}

\bibitem{EGHP}{D. Eisenbud, M. Green, K. Hulek, S. Popescu,
Restricting linear syzygies: algebra and geometry. 
Compos. Math. {\bf 141} (2005), no. 6, 1460--1478. 
}

\bibitem{EHH}{V. Ene, J. Herzog, T. Hibi, Linearly related polyominoes. J. Algebr. Comb. (2015) 41:949--968.}


\bibitem{GL1}
{M. Green and R. Lazarsfeld, The nonvanishing of certain Koszul cohomology groups, J. Differential Geom. {\bf 19} (1984), 168--170.}


\bibitem{GL2}
{M. Green and R. Lazarsfeld, On the projective normality of complete linear series on an algebraic
curve, Invent. Math. {\bf 83} (1985), 73--90.}

\bibitem{GN}
{T. Gulliksen, O. Neg\.ard,
Un complexe r\'esolvant pour certains ideaux d\'eterminantiels. 
C. R. Acad. Sci. Paris Ser. A-B {\bf 274} (1972), A16--A18. }

\bibitem{HKO}{H. T. H\'a, S. Kara,  A O'Keefe, 
Algebraic properties of toric rings of graphs. 
Comm. Algebra {\bf 47} (2019), no. 1, 1--16.}


\bibitem{H}
{M. Hashimoto, Determinantal ideals without minimal free resolutions, Nagoya Math. J. {\bf118} (1990), 203--216.}

\bibitem{HK}
{M. Hashimoto and K. Kurano, Resolutions of determinantal ideals: n-minors of (n+2)-square matrices, Adv. in Math. {\bf 94} (1992), 1--66.}

\bibitem{HHO}
{J. Herzog, T. Hibi, H. Ohsugi, 
Binomial ideals. 
Graduate Texts in Mathematics, 279. Springer, Cham, 2018.}

\bibitem{HMT}
{T. Hibi, K. Matsuda, A. Tsuchiya,  Edge rings with 3-linear resolutions, to appear in Proc. of Amer. Math. Soc., doi: 10.1090/proc/14382}


\bibitem{KP}
{S. Kwak, E. Park, Some effects of property Np on the higher normality and defining equations of nonlinearly normal varieties, J. Reine Angew. Math. {\bf 582} (2005) 87--105}

\bibitem{L}{A. Lascoux, Syzygies des vari\'et\'es determinantales, Adv. Math. {\bf 30} (1978) 202--237.}

\bibitem{M}{M. Mastroeni, Koszul almost complete intersections. 
J. Algebra {\bf 501} (2018), 285--302. }

\bibitem{M2}{D. Grayson, and M. Stillman,
        Macaulay2, a software system for research in algebraic geometry,
       Available at {http://www.math.uiuc.edu/Macaulay2/}
        }

\bibitem{MM}{K. Matsuda, S. Murai, Regularity bounds for binomial edge ideals. J. Commut. Algebra {\bf 5} (2013),
141--149.}

\bibitem{OH1}{H. Ohsugi, T. Hibi, Toric ideals generated by quadratic binomials. J. Algebra {\bf 218} (1999), 509--527.}

\bibitem{OH2}{H. Ohsugi, T. Hibi, Koszul bipartite graphs, Adv. in Appl. Math. {\bf 22} (1999) 25--28.}


\bibitem{OP}{G. Ottaviani, R. Paoletti, Syzygies of Veronese embeddings, Compositio Math. {\bf 125} (2001) 31--37.}

\bibitem{R}
{P. Roberts, 
A minimal free complex associated to the minors of a matrix. (English summary) 
J. Commut. Algebra {\bf 10} (2018), no. 2, 213--242. }

\bibitem{P}{I. Peeva, 
Graded syzygies. 
Algebra and Applications, {\bf 14}, Springer-Verlag London, Ltd., London, 2011.}

\bibitem{PW}{P. Pragacz, J. Weyman, Complexes associated with trace and evaluation. Another approach to Lascoux's resolution, Adv. Math. {\bf 57} (1985) 163--207.}

\bibitem{Q}{A. Qureshi, Ideals generated by 2-minors, collections of cells and stack polyominoes, J.Algebra {\bf 357} (2012)
279--303.}

\bibitem{R}{E. Rubei, 
Resolutions of Segre embeddings of projective spaces of any dimension, 
J. Pure Appl. Algebra {\bf 208} (2007), no. 1, 29--37.}

\bibitem{V}{R.H. Villarreal, \textit{Monomial algebras}. Marcel Dekker, Inc., New York, 2001.}



\end{thebibliography}
\end{document}